\documentclass[11pt]{amsart}
\usepackage{amsfonts}
\usepackage{amsmath,amssymb,latexsym,soul,cite,mathrsfs}
\usepackage{mathtools}
\usepackage{color,enumitem,graphicx}
\usepackage[colorlinks=true,urlcolor=blue,
citecolor=red,linkcolor=blue,linktocpage,pdfpagelabels,
bookmarksnumbered,bookmarksopen]{hyperref}
\usepackage[english]{babel}

\usepackage[left=2.9cm,right=2.9cm,top=2.8cm,bottom=2.8cm]{geometry}
\usepackage[hyperpageref]{backref}

\usepackage[colorinlistoftodos]{todonotes}
\makeatletter
\providecommand\@dotsep{5}
\def\listtodoname{List of Todos}
\def\listoftodos{\@starttoc{tdo}\listtodoname}
\makeatother

\numberwithin{equation}{section}

\makeatletter
\@namedef{subjclassname@2020}{%
	\textup{2020} Mathematics Subject Classification}
\makeatother

\newtheorem{theorem}{Theorem}[section]
\newtheorem{proposition}[theorem]{Proposition}
\newtheorem{lemma}[theorem]{Lemma}
\newtheorem{corollary}[theorem]{Corollary}

\newtheorem{remark}{Remark}

\begin{document}
	
	
	\title[ Existence of positive solution for a class of nonlinear Schr\"odinger equations]{Existence of positive solution for a  class of quasilinear Schr\"odinger equations with potential vanishing at infinity on nonreflexive Orlicz-Sobolev spaces}

	\author[Lucas da Silva]{Lucas da Silva$^\ast$}
	\address{\indent Unidade Acad\^emica de Matem\'atica
		\newline\indent{Universidade Federal de Campina Grande,}
		\newline\indent
		58429-970, Campina Grande - PB - Brasil}
	\email{{ls3@academico.ufpb.br}}

	\author[Marco Souto]{Marco Souto}
	\address{\indent Unidade Acad\^emica de Matem\'atica
		\newline\indent 
		Universidade Federal de Campina Grande,
		\newline\indent
		58429-970, Campina Grande - PB - Brasil}
	\email{{marco@dme.ufcg.edu.br}}

	\pretolerance10000

\thanks{M. Souto was partially supported by CNPq/Brazil 309.692/2020-2}
\thanks{$^\ast$Corresponding author}

	\begin{abstract}
		\noindent In this paper we investigate the existence of positive solution for a class of quasilinear problem on an Orlicz-Sobolev space that can be nonreflexive
		$$
		- \Delta_{\Phi} u +V(x)\phi(|u|)u= K(x)f(u)\mbox{ in } \mathbb{R}^{N},
		$$
		where $ N \geq 2 $, $ V, K $ are nonnegative  continuous functions and $f$ is
		a continuous function with a quasicritical growth. Here we extend the Hardy-type inequalities presented in \cite{AlvesandMarco} to nonreflexive Orlicz spaces. Through inequalities together with a variational method for non-differentiable functionals we will obtain a ground state solution. We analyze also the problem with $V=0$. 
		
	\end{abstract}

	\maketitle
	
	\section{Introduction}

This paper concerns the existence of a positive solution for a class of quasilinear elliptic problem of the type
\begin{equation*}
\left\{\;
\begin{aligned}
-\Delta_{\Phi}u+V(x)\phi(|u|)u=K(x)f(u),\text{ in }\; \mathbb{R}^{N} \\
 u\in D^{1,\Phi}(\mathbb{R}^{N}),\;u\geq 0 ,\text{ in }\; \mathbb{R}^{N}\;\;\;\;\;\;\;
\end{aligned}
\right.
\leqno{(P)}
\end{equation*}
for $ N \geq 2$ and assuming that $V,K:\mathbb R^N \longrightarrow \mathbb R$ and $f:\mathbb{R} \longrightarrow \mathbb R$ are continuous functions with $V$, $K$ being
nonnegative functions and $f$ having a quasicritical growth. It is important to recall that
\begin{equation*}
\Delta_{\Phi}u=\text{div}(\phi(|\nabla u|)\nabla u),
\end{equation*}
where  $\Phi:\mathbb{R}\longrightarrow\mathbb{R}$ is a $N$-function of the form
\begin{equation}\label{*}
\Phi(t)=\int_{0}^{|t|}s\phi(s)ds,
\end{equation}  		
and $\phi:(0,\infty)\longrightarrow(0,\infty)$ is a $C^1$
function verifying some technical assumptions. This type of problem  driven by a $N$-function $\Phi$ appears in a lot of physical
applications, such as Nonlinear Elasticity, Plasticity, Generalized Newtonian Fluid, Non-Newtonian Fluid and Plasma Physics. The reader can find more details about this subject in \cite{DB}, \cite{Figueiredo}, \cite{FN1} and their references.

Recently, motivated by above applications, these types of problems have been studied frequently, we would like to cite Bonanno, Bisci and Radulescu [\citenum{Bonanno1}, \citen{Bonanno2}], Cerny \cite{Cerny}, Clément, Garcia-Huidobro and Manásevich \cite{PH}, Donaldson \cite{Donald}, Fuchs and Li \cite{Li},
Fuchs and Osmolovski \cite{Fuchs}, Fukagai, Ito and Narukawa \cite{FN}, Gossez \cite{Gossez}, Le and Schmitt \cite{Le}, Mihailescu and Radulescu [\citenum{Mihailescu1}, \citenum{Mihailescu2}], Mihailescu and Repovs \cite{Mihailescu3}, Mihailescu, Radulescu and
Repovs \cite{Mihailescu4}, Mustonen and Tienari \cite{Mustonen}, Orlicz \cite{W.Orlicz} and their references. In all of these works, the so-called $\Delta_{2}$-condition was assumed on $\Phi$ and $\tilde{\Phi}$ (Complementary function of $\Phi$), which ensures that the Orlicz-Sobolev space $W^{1,\Phi}(\Omega ) $ and $D ^{1,\Phi} (\Omega)$ are reflexive Banach spaces. This assertion is used several times in order to get a nontrivial solution for elliptic problems taking into account the weak topology and the classical variational methods to $C^{1}$ functionals.

In recent years, problems without the $\Delta_2$-condition of the function $ \tilde{\Phi} $ are being studied. This type of problem brings us many difficulties when we intend to apply variational methods. For example, in our article, the energy functional associated with the problem might not be $C^1$, so classical minimax type results cannot be used here. To get around this difficulty, some recent articles suggest the use of the minimax theory developed by Szulkin \cite{Szulkin}, we cite \cite{AlvesandLeandro}. In this paper, Alves and Carvalho study a class of problem
\begin{equation*}
\left\{\;
\begin{aligned}
-\Delta_{\Phi}u+V(x)\phi(u)u&=f(u),\; \text{ in }\mathbb{R}^{N}& \\
u\in W^{1,\Phi}(\mathbb{R}^{N}) &\text{ with } N\geq 1&
\end{aligned}
\right.
\end{equation*}
when $V$ is $\mathbb{Z}^N$-periodic and $f$ is a continuous function satisfying some technical conditions. The $\Delta_{2}$-condition of $N$-function $\tilde{\Phi}$ has not been required.

Another work in which the $\Delta_{2}$-condition of $N$-function $\tilde{\Phi}$ can be relaxed, is found in \cite{Edcarlos}. Silva, Carvalho, Silva and Gonçalves study a class of problem
\begin{equation*}
\left\{\;
\begin{aligned}
-\Delta_{\Phi}u&=g(x,u),\; \text{ in }\Omega& \\
 u&=0, \text{ on } \partial \Omega &
\end{aligned}
\right.
\end{equation*}
where $\Omega\subset \mathbb{R}^{N}$, $N \geq 2$, is a bounded domain with smooth boundary. 

In this article, to overcome the lack of differentiability of the functional $J$ we present a new approach to work with problems where the $\Delta_{2}$-condition of $N$-function $\tilde{\Phi}$ is relaxed. Here, we will use a weaker version of the mountain pass theorem for functionals that are Gateaux differentiable. Another important point that comes up in this paper is that we cannot use the standard analysis because $D^{1,\Phi}(\mathbb{ R}^N)$ might not be reflexive. This difficulty brings us many problems in order to apply variational methods. In order to overcome these difficulties, we consider the weak$^*$ topology recovering some compactness required in variational methods and one of the results involving the weak$^*$ topology is obtained in Lemma \ref{0.1}.

%
%
Based on the papers above, we assume that $\phi:(0,\infty)\longrightarrow(0,\infty)$ is $C^1$ and satisfies the following hypotheses:
\begin{equation*}
t\longmapsto t\phi(t)\;\text{ is increasing for }\;t>0.
\leqno{(\phi_1)}
\end{equation*}	\vspace{-0.6cm}
\begin{equation*}
\displaystyle\lim_{t\rightarrow0^{+}}t\phi(t)=0 \;\;\text{ and }\;\; \displaystyle\lim_{t\rightarrow+\infty}t\phi(t)=+\infty.
\leqno{(\phi_2)}
\end{equation*}
\begin{equation*}
	1\leq\ell= \inf_{t>0}{\dfrac{\phi(t)t^{2}}{\Phi(t)}}\leq \sup_{t>0}{\dfrac{\phi(t)t^{2}}{\Phi(t)}}=m<N,\;\;\;m < \ell^{*}\text{ and } m\neq1.
\leqno{(\phi_3)}
\end{equation*}	\vspace{-0.3cm}
\begin{equation*}
t\longmapsto\dfrac{\phi(t)}{t^{m-2}} \text{ is nonincreasing for } t > 0.
\leqno{(\phi_4)}
\end{equation*}
\vspace{-0.2cm}

Hereafter, we will say that $\Phi\in\mathcal{C}_m$, if there is a constant $C>0$ satisfying
\begin{equation*}
\Phi(t)\geq C|t|^{m},\;\text{ for all }\;t\in\mathbb{R} .
\leqno{(\mathcal{C}_m)}
\end{equation*}

Next, we show some examples of functions $\Phi$ that can be considered in the present paper.
If $\ell > 1$, we can consider
\begin{enumerate}
	\item[$i)$] $\Phi(t)=|t|^p/p,$ for $p>1$,\vspace*{0.1cm}
	\item[$ii)$] $\Phi(t)=|t|^p/p+|t|^q/q,$ where $1<p<q<N$ with $q\in (p,p^*)$,\vspace*{0.1cm}
	\item[$iii)$] $\Phi(t)=(1+t)^{\alpha}-1,$ $\alpha\in(1,\frac{N}{N-2})$\vspace*{0.1cm}
	\item[$iv)$] $\Phi(t)=t^p\ln(1+t),$ $1< \frac{-1+\sqrt{1+4N}}{2}<p<N-1$, $N\geq 3$.
\end{enumerate}
To the best of our knowledge, there seems to be no legitimate example of a $N$-function that satisfies the conditions $(\phi_{1})-(\phi_{3})$ with $\ell=1$ described in the literature. Lacking an illustrative example, we present the function
\begin{align*}
\Phi_{\alpha}(t)=|t |\ln(|t|^\alpha +1)\; \text{ for } 0<\alpha<\frac{N}{N-1}-1.
\end{align*}
as an example of $N$-function that satisfies $(\phi_{1})-(\phi_{3})$ for the case $\ell=1$. 
%

Here we would like to point out that if $1 < \ell \leq  m < +\infty$,  then function $\Phi$ and its complementary
function $\tilde{\Phi }$ given by
$$\tilde{\Phi}_\alpha(s)=\max_{t\geq 0}\{st-\Phi_\alpha(t)\},~~\text{ for }\; t\geq 0.$$
satisfy the $\Delta_2$-condition. It is well known in the literature that $D^{1,\Phi}(\mathbb{R}^N)$ is reflexive when $\Phi$ and $\tilde{\Phi }$ satisfy the $\Delta_{2}$-condition. Thus, in our paper the space $D^{1,\Phi}(\mathbb{R}^N)$ can be nonreflexive, because we are also considering the case $\ell = 1$

The particular case of the problem $(P)$ with $ \Phi (t) =|t|^2 /2$ (The Laplacian Case) is a study by Alves and Souto \cite{AlvesandMarco}, where the authors show the existence of ground state solutions to obtain a positive solution for the following Schr\"odinger equation 
\begin{equation*}
\left\{\;\begin {aligned}
- \Delta u + V (x) & u = K (x) f (u), & \; \mathbb{R}^{N} \\
& u \in D^{1,2} (\mathbb{R}^{N}) &
\end{aligned}
\right.
\end{equation*}
Also in this paper, the authors impose some conditions under the potentials $V$ and $K$ so that it would be possible to show that the space
\begin{align*}
E = \Big \{u \in D^{1,2} (\mathbb{R}^{N}) / \int_{\mathbb{R}^{N}} V(x) | u |^2dx <+ \infty \Big \} \end{align*}
with norm
\begin{align*}
\lVert u \lVert_{E}^2 =\int_{\mathbb {R}^{N}}(|\nabla u|^2+V(x)|u|^2)dx
\end{align*}
is compactly embedded in the weighted Lebesgue space $ L^{p}_K (\mathbb{R}^{N}) $ for some $ p \in (2, 2^*) $, where
\begin{align*}
L^{p}_K(\mathbb{R}^{N}):= \left \{u: \mathbb{R}^{N} \longrightarrow \mathbb{R}: u\text{ is measurable and} \int_{\mathbb{R}^{N}} K(x) | u |^pdx <+ \infty \right \}
\end{align*}
	
	The purpose of this new paper is to extend the results presented by \cite{AlvesandMarco} to Orlicz-Sobolev space where the conditions can be relaxed. For that, we introduce the following assumptions on the potential $V$ and the coefficient $K$:
\vspace*{0.2cm}\\
\noindent{\it$(K_0)$} $V>0$, $K\in L^{\infty}(\mathbb{R}^{N})$ and $K$ is positive almost everywhere.\vspace*{0.2cm}\\
\noindent{(I)} If $\{A_n\}\subset\mathbb{R}^{N}$ is a sequence of Borelian sets such that $\displaystyle\sup_{n}|A_n|< +\infty$, then
\begin{equation*}
\lim_{r\rightarrow+\infty}\int_{A_n\cap B_r^{c} (0)} K(x) dx=0,\;\text{ uniformly in }n\in\mathbb {N}.
\eqno{(K_1)}
\end{equation*}
\noindent{(II)} One of the below conditions occurs:
\begin{equation*}
\dfrac{K}{V}\in L^{\infty}(\mathbb{ R}^N)
\eqno{(K_2)}
\end{equation*}
or there are $a_1,a_2 \in (m,\ell^*)$ and a $N$-function $A(t)=\int_{0}^{|t|}sa(s) ds$ verifying the following properties:
\begin{align}\label{5.04}
a_1\leq \dfrac{a(t)t^2}{A(t)}\leq a_2
\end{align} 
and 
\begin{equation*}
\textcolor{black}{\dfrac{K(x)}{H(x)}\longrightarrow 0 \text{ when } |x|\rightarrow+\infty}
\eqno{(K_3)}
\end{equation*} 
where $\displaystyle H(x)=\min_{s>0} \left\{V(x)\dfrac{\Phi(s)}{A(s)}+\dfrac{\Phi_{*}(s)}{A(s)} \right\}$.\vspace*{0.2cm}
\newline
\hspace*{0.5cm}Hereafter, we say that $(V, K) \in \mathcal{K}_1$ if $(K_0),(K_1)$ and $(K_2)$ hold. When $(K_0),(K_1)$ and $(K_3)$ hold, then we say that $(V, K) \in \mathcal{K}_2$.

At this point, it is very important to note that $(K_1)$ is weaker than any one of the below conditions:\vspace*{0.1cm}\\
\noindent{${(a)}$} There are $r\geq1$ and $\rho\geq0$ such that $K \in L^r (\mathbb{R}^N \setminus B_{\rho}(0))$;\vspace*{0.1cm}\\
\noindent{${(b)}$} $K(x) \rightarrow 0$ as $|x|\rightarrow \infty$;\vspace*{0.1cm}\\
\noindent{${(c)}$} $K = H_1 + H_2$, with $H_1$ and $H_2$ verifying $(a)$ and $(b)$ respectively.

Now, for every $n\in \mathbb{N }$, fix $z_n=(n,0,\cdots,0)$. Consider $\left\{B_{\frac{1}{2^n}}(z_n)\right\}_{n\in\mathbb{N}}$ the disjoint sequence of open balls in $\mathbb{R}^{N}$ and the nonnegative function $ H_1:\mathbb{R}^{N}\longrightarrow\mathbb{R}$ given by
{\small\begin{align*}
0\leq H_1(x)\leq 1,\;\forall x\in\mathbb{R}^{N},\;\;H_1(z_n)=1,\;\forall n\in\mathbb{N},\;\;\;	H_1\equiv 0\;\text{ in }\;\mathbb{R}^{N}\setminus\Big\{\bigcup_{n\in\mathbb{N}}B_{\frac{1}{2^n}}(z_n)\Big\}
\end{align*}}
and 
\begin{align*}
\int_{B_{\frac{1}{2^n}} (z_n)}H_1(x)dx\leq \dfrac{1}{2^n},\;\;\;\forall n\in\mathbb{N} .
\end{align*}
Without any difficulties, we can see that the functions
\begin{align}\label{00.0}
V(x)=K(x)=H_1(x)+\dfrac{1}{\ln(2+|x|)}
\end{align}
satisfies the item $(c)$ and consequently verifies the condition $(K_1)$. In addition, clearly $V$ and $K$ satisfy the conditions $(K_0)$ and $(K_2)$. However, these functions do not verify the condition $(K_3)$.

Now consider the functions
\begin{equation}\label{00.1}
{K(x)=H_1(x)+\dfrac{1}{ln(2+|x|)}}
\end{equation}
and
\begin{align}\label{00.2}
\begin{split}
V(x)=\left(|x|H_1(x)\right)^{\frac{m^*-m}{m^*-a_1}}+\left(\dfrac{1}{ln(2+|x|)}\right)^{\frac{m^*-m}{m^*-a_2}}+\left(|x|H_1(x)\right)^{\frac{\ell^*-\ell}{\ell^*-a_2}}+\left(\dfrac{1}{ln(2+|x|)}\right)^{\frac{\ell^*-\ell}{\ell^*-a_1}}
\end{split}
\end{align}
As mentioned before, these functions verifies $(K_0)$ and $(K_1)$. Furthermore, the condition $(K_3)$ is satisfied. However, these functions do not verifies the condition $(K_2)$.
		

                                                                                                                                                           
This paper is organized as follows: in Section $2$, we review some properties of Orlicz and Orlicz-Sobolev spaces that will be used throughout this paper.

In Section $3$, we present the space and show some important properties involving the energy functional related to the problem $(P)$.

In Section 4, we will study the case $(V,K)\in \mathcal{K}_1$. In our first main result, by means of some conditions imposed on $\Phi$ and $f$, we will show that the problem $(P)$ has a $C^{1,\alpha}_{loc}(\mathbb{R}^N)$ positive ground state solution. More specifically, we will assume that $\Phi\in \mathcal{C}_m$ and $f:\mathbb{R} \longrightarrow\mathbb{R}$ satisfies the following conditions
\begin{equation*}
	\displaystyle \lim_{t \to 0} \dfrac{f(t)}{ t\phi(t)}=0 \;\;\; \text{ and }\;\;\;\displaystyle\limsup_{t \to \infty} \dfrac{f(t)}{t \phi_{*}(t)}=0.
	\leqno{(f_1)}
\end{equation*}
%
\noindent{\it$(f_2)$} $s^{1-m}f(s)$ is an increasing function in $(0,+\infty)$.\vspace*{0.1cm}\\
\noindent{\it$(f_3)$} $F(t)=\int_{0}^{t}f(s)ds$ is $m$-superlinear at infinity, that is,
\begin{equation*}
\lim_{|t|\rightarrow+\infty} \dfrac{F(t)}{|t|^{m}}=+\infty.
\end{equation*}

%

%
%
%
%
Under these conditions, our first main result can be stated as follows.

\begin{theorem}\label{Teo1}
	 Assume that $(V,K)\in \mathcal{K}_1$ and $\Phi\in \mathcal{C}_m$. Suppose that $(\phi_{1})- (\phi_{4})$ and $(f_{1})-(f_{3})$ hold. Then problem $(P)$ possesses a nonnegative solutions that are locally bounded.
\end{theorem}

To study the regularity of the solutions provided by Theorem \ref{Teo1}, we add the following assumptions:

\noindent{\it$(\phi_5)$} There are $0<\delta<1$, $C_1, C_2>0$ and $1<\beta\leq\ell^*$ such that 
\begin{equation*}
C_1t^{\beta -1}\leq t\phi(t)\leq C_2t^{\beta -1},\;\text{ for }\;t\in[0,\delta].
\end{equation*}
\noindent{\it$(\phi_6)$} There are constants $\delta_0>0$ and $\delta_1>0$ such that 
\begin{equation*}
\delta_0\leq\dfrac{(\phi(t)t)'}{\phi(t)}\leq \delta_1\;\text{ for}\;t>0.
\end{equation*}

We are in position to state the following regularity result:

\begin{theorem}\label{Teo11} Supoose that $\Phi$ satisfies $(\phi_{5})- (\phi_{6})$.
	Under the assumptions of Theorem \ref{Teo1}, the problem $(P)$ possesses a $C^{1,\alpha}_{loc}(\mathbb{R}^N)$ positive ground state solution.
\end{theorem}

For the second main result, we consider the problem $(P)$ without condition $\mathcal{C}_m$. The removal of this condition on the $N$-function $\Phi$, forced us to present a more restricted growth condition that $(f_1)$, this constraint under the nonlinearity $f$, will be necessary to show that the nonnegative solutions of $(P)$, are positive. In this way, we will consider $B:\mathbb{R}\rightarrow[0,\infty)$ being a $N$-function given by
$B(t) =\int_{0}^{|t|} b(s)s ds$, where $b : (0, \infty) \rightarrow (0,\infty)$ is a function satisfying the following conditions:
\begin{equation*}
t\longmapsto tb(t)\;\text{ is increasing for }\;t>0,
\leqno{(B_1)}
\end{equation*}
\begin{equation*}
\displaystyle\lim_{t\rightarrow0^{+}}tb(t)=0 \;\;\text{ and }\;\; \displaystyle\lim_{t\rightarrow+\infty}tb(t)=+\infty,
\leqno{(B_2)}
\end{equation*}
and there exist $b_1 \in [m,\ell^*]$ such that\vspace*{-0.2cm}
\begin{equation*}
b_1=\inf_{t>0}\dfrac{b(t)t^2}{B(t)} \;\;\text{ and }\;\; \ell^*\geq\sup_{t>0}\dfrac{b(t)t^2}{B(t)},\;\;\forall t>0.\vspace*{-0.2cm}
\leqno{(B_3)}
\end{equation*} 
For this case, we assume that  $f:\mathbb{R} \longrightarrow\mathbb{R}$ satisfies the conditions $(f_2)$ and $(f_3)$. Moreover, we will consider the following growth condition
\begin{equation*}
\displaystyle \lim_{t \to 0} \dfrac{f(t)}{ t\phi(t)}=0 \;\;\; \text{ and }\;\;\;\displaystyle\limsup_{t \to \infty} \dfrac{|f(t)|}{t b(t)}=0.
\leqno{(f_4)}
\end{equation*}

Our second main result can be written in the following form.
\begin{theorem}\label{Teo1.1}
	Assume that $(V,K)\in \mathcal{K}_1$. Suppose that $(\phi_{1})- (\phi_{4})$, $(f_{2})$, $(f_{3})$ and $(f_{4})$ hold.  Then problem $(P)$ possesses a nonnegative solutions that are locally bounded. If $\Phi$ also satisfies $(\phi_{5})$ and $(\phi_{6})$, the solutions for the problem $(P)$ are $C^{1,\alpha}_{loc }(\mathbb{R}^N)$ positive ground state solution.
\end{theorem}

In Section 5, we will study the case that $(V,K)\in \mathcal{K}_2$. As in section 4, we initially analyzed the problem $(P)$ with $\Phi\in \mathcal{C}_m$. In this sense, we assume that $f:\mathbb{R} \longrightarrow\mathbb{R}$ satisfies $(f_2)$ and $(f_3)$. Moreover, we will consider the following condition
\begin{equation*}
\displaystyle \limsup_{t \to 0} \dfrac{|f(t)|}{ ta(t)}<\infty \;\;\; \text{ and }\;\;\;\displaystyle\lim_{t \to \infty} \dfrac{f(t)}{t \phi_{*}(t)}=0.
\leqno{(f_5)}
\end{equation*}

Our first main result of this fifth section can be stated as follows.
\begin{theorem}\label{Teo3}
	Assume that $(V,K)\in \mathcal{K}_2$ and $\Phi\in \mathcal{C}_m$. Suppose that $(\phi_{1})- (\phi_{4})$, $(f_{2})$, $(f_{3})$ and $(f_{5})$ hold.
	  Then problem $(P)$ possesses a nonnegative solutions that are locally bounded. If $\Phi$ also satisfies $(\phi_{5})$ and $(\phi_{6})$, the solutions for the problem $(P)$ are $C^{1,\alpha}_{loc }(\mathbb{R}^N)$ positive ground state solution.
\end{theorem}

In a second moment, as in Theorem \ref{Teo1.1}, we relax the condition $\Phi\in \mathcal{C}_m$ and present a more restricted growth condition that $(f_5)$. More precisely, $f:\mathbb{R} \longrightarrow\mathbb{R}$ satisfies  
\begin{equation*}
\displaystyle \limsup_{t \to 0} \dfrac{|f(t)|}{t a(t)}<\infty \;\;\; \text{ and }\;\;\;\displaystyle\lim_{t \to \infty} \dfrac{f(t)}{t b(t)}=0.
\leqno{(f_6)}
\end{equation*}
Furthermore, we will assume that $f$ satisfies $(f_2)$ and $(f_3)$.

Our second main result of this fifth section can be written in the following form.

\begin{theorem}\label{Teo3.1}
	Assume that $(V,K)\in \mathcal{K}_2$. Suppose that $(\phi_{1})- (\phi_{4})$, $(f_{2})$, $(f_{3})$ and $(f_{6})$ hold.  Then problem $(P)$ possesses a nonnegative solutions that are locally bounded. If $\Phi$ also satisfies $(\phi_{5})$ and $(\phi_{6})$, the solutions for the problem $(P)$ are $C^{1,\alpha}_{loc }(\mathbb{R}^N)$ positive ground state solution.
\end{theorem}

In Section 6, we will study the zero mass case, this is, $V=0$. In what follows, we assume that $f$ satisfies $(f_2)$ and $(f_3)$. Moreover, we will consider the following growth condition
\begin{equation*}
\displaystyle \lim_{t \to 0} \dfrac{f(t)}{ t\phi_{*}(t)}=0 \;\;\; \text{ and }\;\;\;\displaystyle\lim_{t \to \infty} \dfrac{f(t)}{t\phi_{*}(t)}=0.
\leqno{(f_7)}
\end{equation*}

Our main result stated in this section is described next.
\begin{theorem}\label{Teo2}
	 Suppose that $(\phi_{1})- (\phi_{4})$ and $(K_0)-(K_1)$ hold. Assume that $f$ satisfies $(f_2)$, $(f_{3})$, $(f_{7})$. If $V\equiv0$, then the problem $(P)$ has a nonnegative ground state solution.
\end{theorem}


\section{Basics On Orlicz-Sobolev Spaces}
In this section we recall some properties of Orlicz and Orlicz-Sobolev spaces, which can be found in [\citenum{Adms},\citenum{FN},\citenum{Rao},\citenum{Tienari}]. First of all, we recall that a continuous function $\Phi:\mathbb{R}\rightarrow [0,+\infty)$ is a $N$-function if:
\begin{itemize}
	\item[(i)] $\Phi$ is convex;
	\item[(ii)] $\Phi(t)=0\Leftrightarrow t=0$;
	\item[(iii)] $\Phi$ is even;
	\item[(iv)] $\displaystyle\lim_{t\rightarrow 0}\dfrac{\Phi(t)}{t}=0\text{ and }\lim_{t\rightarrow +\infty}\dfrac{\Phi(t)}{t}=+\infty$.
\end{itemize}\vspace*{0.2cm}

We say that a $N$-function $\Phi$ verifies the $\Delta_2$-condition, and we denote by $\Phi\in(\Delta_2)$, if there are constants $K>0,\; t_{0}>0$ such that $$\Phi(2t)\leq K\Phi(t),~~\forall t \geq t_0.$$ In the case of $|\Omega|=+\infty$, we will consider that $\Phi\in(\Delta_2)$ if $t_0=0$. For instance, it can be shown that $\Phi(t)=|t|^p /p$ for $p > 1$ satisfies the $\Delta_{2}$-condition, while $\Phi(t)=(e^{t^2}-1)/2$. 

If $\Omega$ is an open set of $\mathbb{R}^N$, where $N$ can be a natural number such that $N\geq 1$, and $\Phi$ a $N$-function then define the Orlicz space associated with $\Phi$ as $$L^{\Phi}(\Omega)=\left\{u\in L^1_{\text{loc}}(\Omega):~\int_\Omega \Phi\left(\frac{|u|}{\lambda}\right) dx<+\infty~\text{for some}~\lambda>0\right\}.$$ The space $L^{\Phi}(\Omega)$ is a Banach space endowed with the Luxemburg norm given by $$\|u\|_{L^{\Phi}(\Omega)}=\inf\left\{\lambda>0:\int_\Omega \Phi\left(\frac{|u|}{\lambda}\right) dx\leq 1\right\}.$$ In the case that $\Phi$ verifies $\Delta_2$-condition we have $$L^{\Phi}(\Omega)=\left\{u\in L^1_{\text{loc}}(\Omega):~\int_\Omega \Phi(|u|) dx<+\infty\right\}.$$
The complementary function $\tilde{\Phi}$ associated with $\Phi$ is given by the Legendre transformation, that is, $$\tilde{\Phi}(s)=\max_{t\geq 0}\{st-\Phi(t)\},~~\forall\; t\geq 0.$$
The functions $\Phi$ and $\tilde{\Phi}$ are complementary each other and satisfy the inequality below $$ \tilde{\Phi}(\Phi'(t))\leq \Phi(2t),\;\;\forall\; t>0.$$ Moreover, we also have a Young type inequality given by $$st\leq \Phi(t)+\tilde{\Phi}(s),~~~\forall s,t\geq 0.$$ Using the above inequality, it is possible to establish the following Holder type inequality: $$\left|\int_{\Omega}uvdx\right|\leq 2\|u\|_{L^{\Phi}(\Omega)}\|v\|_{{L^{\tilde{\Phi}}(\Omega)}},~~\text{for all}~~u\in L^{\Phi}(\Omega)~~ \text{and}~~ v\in L^{\tilde{\Phi}}(\Omega).$$

If $|\Omega|<\infty$, the space $E^{\Phi}(\Omega)$ denotes the closing of $L^{\infty}(\Omega)$ in $L^{\Phi}( \Omega)$ with respect to the norm $\lVert \cdot\lVert_{\Phi}.$ When $|\Omega|=\infty$, the space $E^{\Phi}(\Omega)$ denotes the closure of $C^{\infty}_{0}(\Omega)$ in $L^{\Phi}(\Omega)$ with respect to norm $\lVert \cdot\lVert_{\Phi}.$ In any of these cases, $L^{\Phi}(\Omega)$ is the dual space of $E^{\tilde{\Phi}}(\Omega)$, while $L^{\tilde{\Phi}}( \Omega)$ is the dual space of $E^{{\Phi}}(\Omega)$.  Moreover, $E^{{\Phi}}(\Omega)$ and $E^{\tilde{\Phi}}(\Omega)$ are separable and all continuous functional $M:E^{{\Phi} }(\Omega)\longrightarrow\mathbb{R}$ is of the form
\begin{align*}
M(v)=\int_{\Omega} v(x)g(x)dx,\;\;\;\;\text{for some function}\;\; g\in L^{\tilde{\Phi}}(\Omega).
\end{align*}

Another important function related to function $\Phi$, it is the Sobolev conjugate function ${\Phi}_*$
of $\Phi$ defined by
\begin{align*}
\Phi_{*}^{-1}(t)=\int_{0}^{t}\dfrac{\Phi^{-1}(s)}{s^{(N+1)/N}}ds\;\;\text{for}\;\;t>0\;\;\text{ when }\;\;\int_{1}^{+\infty}\dfrac{\Phi^{-1}(s)}{s^{(N+1)/N}}ds=+\infty.
\end{align*}

\begin{lemma}\label{0.3}
	Consider $\Phi$ a $N$-function of the form \eqref{*} and satisfying $(\phi_1)$ and $(\phi_2)$. Set $$\xi_0(t)=\min\{t^{\ell},t^{m}\}~~\text{and}~~\xi_1(t)=\max\{t^{\ell},t^{m}\},~~\forall t\geq 0.$$ Then $\Phi$ satisfies $$\xi_0(t)\Phi(\rho)\leq\Phi(\rho t)\leq\xi_1(t)\Phi(\rho),~~\forall \rho,t\geq0$$ and $$
	\xi_0(\|u\|_{\Phi})\leq \int_\Omega \Phi(u)dx\leq \xi_1(\|u\|_{\Phi}), \;\;\;\forall u \in L^{\Phi}(\Omega).$$
\end{lemma}

\begin{lemma}\label{0.5}
	If $\Phi$ is an N-function of the form \eqref{*} satisfying $(\phi_1)$, $(\phi_2)$ and $(\phi_{3})$, then
	\begin{equation}\label{0.22}
		\xi_2 (t) \Phi_*(\rho)\leq \Phi_*(\rho t) \leq \xi_{3}(t)\Phi_*(\rho),\;\;\; \forall\rho, t >0
	\end{equation}
	and
	\begin{align*}
	\xi_{2}(\Arrowvert u \Arrowvert_{\Phi_*}) \leq \int_{\Omega} \Phi_*(u)dx \leq \xi_{3} (\Arrowvert u\Arrowvert_{\Phi_* }), \;\;\;\forall u \in L^{\Phi_*}(\Omega),
	\end{align*}
	where $$ \xi_2 (t) = \min \{t^{\ell^*} , t^{m^*}\} \hspace{0.1cm} \text{ and } \hspace{0.1cm} \xi_{3}(t) = \max\{t^{\ell^*}, t^{m^*}\}, \;t \geq 0.$$
\end{lemma}

\begin{lemma}\label{}
	\eqref{0.22} is equivalent to
		\begin{equation*}
{\ell^*}\leq \dfrac{\Phi'_*(t)t^2}{\Phi_*(t)} \leq {m^*},\;\;\; \forall t >0
	\end{equation*}
\end{lemma}

\begin{lemma}\label{0.51}
	If $\Phi$ is an N-function of the form \eqref{*} satisfying $(\phi_1)$, $(\phi_2)$ and $(\phi_{3})$, then
	\begin{equation*}
\tilde{\Phi } (\rho t) \leq t^{\frac{m}{m-1}}\tilde{\Phi }(\rho),\; \text{ for all }\;\rho>0\;\text{ and }\; 0\leq t<1.
	\end{equation*}
\end{lemma}

\begin{lemma}
If $\Phi$ is $N$-function and $ (\int_{\Omega}\Phi (|u_n|)dx)$ is a bounded sequence, then $(u_n)$ is a bounded sequence in $L^{\Phi}(\Omega)$. When $\Phi\in(\Delta_{2})$, the equivalence is valid.
\end{lemma}


\begin{lemma}
		If $\Phi$ is an N-function of the form \eqref{*} satisfying $(\phi_1)$, $(\phi_2)$ and $(\phi_{3})$ with $\ell =1$, then $\tilde{\Phi }\notin (\Delta_{2})$.
\end{lemma}

\begin{lemma}\label{1.0}
	Let $\Omega\subset\mathbb{R}^{N}$ be a domain, if $\Phi\in (\Delta_{2})$ and $(u_n)$ a sequence in $L^{\Phi}(\Omega)$ with $u_n\longrightarrow u$ in $L^{\Phi}(\Omega)$, there is $H\in L^{\Phi}(\Omega)$ and a subsequence $(u_{n_j})$ such that\vspace*{0.2cm}\\
	$i)\;\;\;|u_{n_j}(x)|\leq H(x)$ a.e. in $\Omega$.\vspace*{0.2cm}\\
 $ii)\;\;\;u_{n_j}(x)\longrightarrow u(x)$ a.e. in $\Omega$ and all $j\in\mathbb{N}$.
\end{lemma}

For a $N$-function $\Phi$, the corresponding Orlicz-Sobolev space is defined as the Banach space $$W^{1, \Phi}(\Omega)=\left\{ u \in L^ {\Phi}(\Omega): \dfrac{\partial u}{\partial x_{i}} \in L^{\Phi}(\Omega), i=1,...,N \right\} ,$$ with the norm
\begin{align}\label{00}
	\lVert u\Arrowvert_{1, \Phi} = \Arrowvert \nabla u \Arrowvert_{\Phi} + \Arrowvert u \Arrowvert_{\Phi}.
\end{align}
%

If $\Phi\in(\Delta_{2})$, the space $D^{1,\Phi}(\mathbb{R}^{N})$ is defined to be the complement of the space $C^{ \infty}_{0}(\mathbb{R}^{N})$ with respect to the standard
\begin{align}\label{0.10}
| u|_{D^{1,\Phi}(\mathbb{R}^{N})}=\lVert u\lVert_{\Phi_{*}}+\lVert \nabla u\lVert_{\Phi}.
\end{align}
Since the Orlicz-Sobolev inequality
	\begin{align}\label{0.9}
	\lVert u\lVert_{\Phi_{*}}\leq S_N\lVert \nabla u\lVert_{\Phi},
	\end{align}
holds for $u\in D^{1,\Phi}(\mathbb{R}^{N})$ with a constant $S_N>0$, the norm \eqref{0.10} is equivalent to the norm
\begin{align}\label{0.11}
\| u\|_{D^{1,\Phi}(\mathbb{R}^{N})}=\lVert \nabla u\lVert_{\Phi},
\end{align}
on $D^{1,\Phi}(\mathbb{R}^{N})$. In this paper, we will use \eqref{0.11} as the norm of $D^{1,\Phi}(\mathbb{R}^{N})$. Clearly $$D^{1,\Phi}(\mathbb{R}^{N})\xhookrightarrow[cont\;]{} L^{\Phi_{*}}(\mathbb{R}^{N}).$$
\vspace*{-0.1cm}
The space $L^{\Phi}(\mathbb{R}^{N})$ is separable and reflexive when the $N$-functions $\Phi$ and $\tilde{\Phi}$ satisfy the $\Delta_{2}$-condition. Knowing that $D^{1,\Phi}(\mathbb{R}^{N})$ can be seen as a closed subspace of the space $L^{\Phi_{*}}(\mathbb{R}^{ N})\times (L^{\Phi}(\mathbb{R}^{N}))^{N}$, then $D^{1,\Phi}(\mathbb{R}^{N} )$ is reflexive when the $N$-functions $\Phi$, $\tilde{\Phi}$, $\Phi_*$ and $\tilde{\Phi}_*$ satisfy the $\Delta_{2}$-condition .

The following lemma is an immediate consequence of the Banach-Alaoglu-Bourbaki theorem \cite{Brezis}.

\begin{lemma}\label{0.6}
	Assume that $\Phi$ is an N-function of the form \eqref{*} satisfying $(\phi_1)$, $(\phi_2)$ and $(\phi_{3})$. If $(u_n)\subset D^{1,\Phi}(\mathbb{R}^{N})$ is a bounded sequence, then there exists a subsequence of $(u_n)$, which we will still denote by $(u_n )$, and $u\in D^{1,\Phi}(\mathbb{R}^{N})$ such that
	\begin{align}\label{0.23}
	u_n\xrightharpoonup[\quad]{\ast} u\;\;\;\text{ in }\;L^{\Phi_{*}}(\mathbb{R}^{N})\;\;\;\;\text{ and }\;\;\;\;	\dfrac{\partial u_n}{\partial x_i}\xrightharpoonup[\quad]{\ast} \dfrac{\partial u}{\partial x_i}\;\;\;\text{ in }\;L^{\Phi}(\mathbb{R}^{N})
	\end{align}
	or equivalently,
	\begin{align*}
	\int_{\mathbb{R}^N}u_nvdx\longrightarrow \int_{\mathbb{R}^{N}}uvdx,\;\;\forall v\in E^{\tilde{\Phi}_*}(\mathbb{R}^N)\end{align*}
and
	\begin{align*}
	\int_{\mathbb{R}^{N}}\dfrac{\partial u_n}{\partial x_i}wdx\longrightarrow\int_{\mathbb{R}^{N}}\dfrac{\partial u}{\partial x_i}wdx,\;\;\forall w \in E^{\tilde{\Phi}}(\mathbb{R}^{N}).
	\end{align*}
\end{lemma}

From now on, we denote the limit \eqref{0.23} by $u_n\xrightharpoonup[\quad]{\ast} u$ in $D^{1,\Phi}(\mathbb{R}^{N})$.
As an immediate consequence of the last lemma, we have the following corollary.

\begin{corollary}\label{0.8}
	If $(u_n)\subset D^{1,\Phi}(\mathbb{R}^{N})$ is a bounded sequence with $u_n\longrightarrow u$ in $L^{\Phi}_{loc}(\mathbb{R}^{N})$, then $ u\in D^{1,\Phi}(\mathbb{R}^{N})$.
\end{corollary}

The last lemma is crucial when the space $D^{1,\Phi}(\mathbb{R}^{N})$ can be nonreflexive. And, this happens, for example, when $\Phi_{\alpha}(t)=|t|\ln(|t|^\alpha+1)$, for  $0<\alpha<\frac{N}{N-1}-1$, because $\tilde{\Phi }_{\alpha}$ does not verify the $\linebreak\Delta_{2}$-condition. Here we emphasize that the condition $(\phi_{3})$ guarantees that $\Phi$ and $\tilde{\Phi}$ verifies the $\Delta_{2}$-condition when $\ell> 1$, for more details see Fukagai and Narukawa \cite{FN}.

\section{Preliminary results}
 This section focuses on preparing some preliminaries for proving Theorems \ref{Teo1}, \ref{Teo1.1}, \ref{Teo3} and \ref{Teo3.1}. Since the potential V may vanish at infinity, we cannot study equation $(P)$ on the Sobolev space $D^{1,\Phi}(\mathbb{R }^N)$ by variational methods. As
 in \cite{AlvesandMarco}, we work in the space $\linebreak E=\big\{u\in D^{1,\Phi}(\mathbb{R}^{N}): \int_{\mathbb{R}^{N}} V(x)\Phi(|u |)dx<+\infty\big\}$ with norm
\begin{align*}
\lVert u\lVert_{E}=\lVert u\lVert_{D^{1,\Phi}(\mathbb{R}^{N})} +\lVert u\lVert_{V,\Phi},
\end{align*}
where
\begin{align*}
\lVert u\lVert_{V,\Phi}=\inf\Big\{\alpha >0: \int_{\mathbb{R}^{N}} V(x)\Phi\left({| u|}/{\alpha}\right)dx\leq1\Big\}
\end{align*}
is the norm of Banach Space
\begin{align*}
L^{\Phi}_{V}(\mathbb{R}^{N})=\Big\{u:\mathbb{R}^{N}\longrightarrow\mathbb{R}\;\text{measurable }:\int_{\mathbb{R}^{N}}V(x)\Phi(|u|)dx<+\infty\Big\}.
\end{align*}
It is immediate that $E$ is continuously embedded in the spaces $D^{1,\Phi}(\mathbb{R}^{N}) $ and $L^{\Phi}_{V}(\mathbb{R }^{N})$.

Now let us list some properties involving the space $E$.

\begin{lemma}
	$(E,\lVert \cdot\lVert_{E})$ is a Banach space.
\end{lemma}

\begin{lemma}
	$E=\overline{C^{\infty}_{0}(\mathbb{R}^{N})}^{\lVert\cdot\lVert_{E}}$
\end{lemma}
\noindent{\it{Proof:}}
It follows the same ideas as Theorem 8.21, which can be found in \cite{Adms}. For that reason, we will omit your proof.

\qed

\begin{lemma}\label{0.12}
	$E$ is compactly embedded in $L^{\Phi}_{loc}(\mathbb{R}^{N})$.
\end{lemma}

\begin{lemma}\label{0.1}
	Suppose $(u_n)\subset E$ is a bounded sequence, then there is $u\in E$ such that $u_n \xrightharpoonup[\quad]{\ast} u$ in $D^{1,\Phi}(\mathbb{R}^{N})$.
\end{lemma}
\noindent {\it{Proof:}}
Since $(u_n)$ is a bounded sequence in $E$, then $(u_n)$ is a bounded sequence in $D^{1,\Phi}(\mathbb{R}^{N}) $ and by the Lemma \ref{0.6}, there is $ u\in D^{1,\Phi}(\mathbb{R}^{N}) $ such that $ u_n\xrightharpoonup[\quad]{\ast} u $ in $D^{1,\Phi}(\mathbb{R}^{N})$. Let us show that $u\in E$, because by Lemma \ref {0.12} and Corollary \ref {0.8}, we can conclude that less than one subsequence
\begin{align*}
u_n (x) \longrightarrow u(x), \; \; \; \; a.e. \; \; \; \text{in} \; \; \; \mathbb{R}^{N}.
\end{align*}
By Fatou's Lemma
\begin{align*}
\int_{\mathbb{R}^{N}} V(x) \Phi(|u(x)|)dx \leq \liminf_{n \rightarrow \infty} \int_{\mathbb{R}^{N }} V(x) \Phi(|u_{n}(x)|)dx.
\end{align*}
Since $(u_n)$ is bounded in $E$, then $(u_n)$ is bounded in $L^{\Phi}_{V}(\mathbb{R}^{N})$. As $\Phi\in(\Delta_{2})$, there is $C>0$ such that
\begin{align*}
\int_{\mathbb{R}^{N}} V(x) \Phi(|u_{n}(x)|)dx\leq C,\;\;\;\;\forall n\in\mathbb{N}
\end{align*}
Therefore,
\begin{align*}
\int_{\mathbb{R}^{N}}V(x)\Phi(|u(x)|)dx<+\infty,
\end{align*}
showing that $u\in E$, and the proof is complete.

\qed

Now, we consider the functional $Q:E\longrightarrow\mathbb{R}$ which is given by
$$
Q(u)=\int_{\mathbb{R}^{N}} \Phi(|\nabla u|)dx+\int_{\mathbb{R}^{N}}V(x)\Phi(|u|)dx.
$$
It is well known in the literature that $Q\in C^1(E,\mathbb R)$ when $\Phi$ and $\tilde{\Phi}$ satisfy the condition $(\Delta_2)$ and this occurs when we have the condition satisfied to $\ell>1$. When $\ell=1$, we know that $\tilde{\Phi}\notin (\Delta_{2})$ and therefore cannot guarantee the differentiability of functional $Q$. However, we will show that the functional $Q$ is continuous and Gateaux-differentiable with derivative $Q': E\longrightarrow E^*$ defined by
\begin{align*}
Q'(u)v=\int_{\mathbb{R}^{N}} \phi(|\nabla u|)\nabla u\nabla v dx+\int_{\mathbb{R}^{N} } V(x)\phi(|u|)uv dx,\;\;\;\forall u,v\in E
\end{align*}
is continuous from the norm topology of $E$ to the weak$^*$-topology of $E^*$.  



\begin{lemma}
	The functional $Q$ is Gateaux differentiable, that is, $ Q'(u)v$ exists for every $u,v\in E$ with
	\begin{align*}
	Q'(u)v=\int_{\mathbb{R}^{N}} \phi(|\nabla u|)\nabla u\nabla v dx+\int_{\mathbb{R}^{N} } V(x)\phi(|u|)uv dx.
	\end{align*}
\end{lemma}
\noindent {\it{Proof:}}
See Lemma 4.1 in \cite{AlvesandLeandro}.

\qed
\begin{lemma}\label{1.1}
	Let $\Phi$ an N-function of the form \eqref{*} satisfying $(\phi_1)$, $(\phi_2)$ and $(\phi_{3})$. If $u_n\longrightarrow u$ in $L^{\Phi}_{V}(\mathbb{R}^{N})$, then there exists $H_1\in L^{\Phi}_{V}(\mathbb{R}^{N})$ and a subsequence $ \{u_{n_j}\}$ such that\vspace*{0.2cm}\\
	$i)\;\;|u_{n_j}(x)|\leq H_1(x)$ for every $x\in\mathbb{R}^{N}$ and every $j\in\mathbb{N}$\vspace*{0.2cm}\\
	$ii)\;\;u_{n_j}(x)\longrightarrow u(x)$ $a.e.$ in $\mathbb{R}^{N}$ and every $j\in\mathbb{N}$.
\end{lemma}
\noindent {\it{Proof:}} It is enough to repeat the same argument explored in [\citenum{Tienari}, Lemma $2.5$].

\qed

As an immediate consequence of the Lemma \ref{1.1}, we have the following result.

\begin{lemma}
	The functional $Q:E\longrightarrow\mathbb{R}$ is continuous in the norm topology.
\end{lemma}

\begin{lemma}\label{Lema1}
	The Gateaux derivative $Q':E\longrightarrow E^*$ is continuous from the norm topology of $E$ to the weak$^*$ topology of $E^*$.
\end{lemma}
\noindent {\it{Proof:}} By the Proposition $3.2$ of \cite{Brezis} is sufficient prove that, any sequence $(u_n)\subset E $ such that $u_n\longrightarrow u$ in $E$, implies
$$\langle Q'(u_n),v\rangle\longrightarrow \langle Q'(u),v\rangle,\;\;\;\forall v\in E.$$

Consider $(u_n)\subset E$ such that $u_n\longrightarrow u$ in $E$, then
\begin{equation*}
|\nabla u_n|\longrightarrow|\nabla u|\;\;\text{ in }\;\;L^{\Phi}(\mathbb{R}^{N})\;\;\text{ and }\;\;
u_n\longrightarrow u \;\text{ in }\;L^{\Phi}_{V}(\mathbb{R}^{N}).
\end{equation*}
By Lemma \ref{1.1}, there are $H_1\in L^{\Phi}_{V}(\mathbb{R}^{N})$, $H_2\in L^{\Phi}_{V}(\mathbb{R}^{N})$ and a subsequence $\{u_{n_j}\}\subset\{u_n\}$ such that \\
{\it{i)}} $| u_{n_j}(x)|\leq H_1(x),\;$ for $x\in\mathbb{R}^{N}$ and $j\in\mathbb{N}$ \\
{\it{ii)}} $| \nabla u_{n_j}(x)|\leq H_2(x),\;$ for $x\in\mathbb{R}^{N}$ and $j\in\mathbb{N}$ \\
{\it{iii)}} $u_{n_j}(x) \rightarrow u(x),\;$ $a.e.$ in $ \mathbb {R}^{N} $ and $j\in\mathbb{N}$ \\
{\it{iv)}} $|\nabla u_{n_j}(x)|\longrightarrow|\nabla u(x)|,\;$ $a.e.$ in $ \mathbb {R}^{N} $ and $j\in\mathbb{N}$.

Set $v\in E$ arbitrary. By the continuity of the function $\phi$, it follows that
\begin{equation*}
\phi(|\nabla u_{n_j}(x)|)\nabla u_{n_j}(x)\nabla v(x)\longrightarrow	\phi(|\nabla u(x)|)\nabla u(x)\nabla v(x),\;a.e. \text{ in } \mathbb {R}^{N}.
\end{equation*}
Also, by $(\phi_1)$ the function $\phi(t)t$ is increasing for every $t>0$, thus
\begin{align*}
\phi(|\nabla u_{n_j}(x)|)\nabla u_{n_j}(x)\nabla v(x)|\leq 	\phi(|\nabla u_{n_j}(x)|)|\nabla u_{n_j}(x)||\nabla v(x)|\leq 	\phi(|H(x)|)| H(x)|\nabla v(x)|.
\end{align*}
Hence by Lebesgue dominated convergence theorem
\begin{align*}
\int_{\mathbb{R}^{N}}	\phi(|\nabla u_{n_j}|)\nabla u_{n_j}\nabla v dx\longrightarrow
\int_{\mathbb{R}^{N}}	\phi(|\nabla u|)\nabla u\nabla v dx,
\end{align*}
with this, we ensure that
\begin{align*}
\int_{\mathbb{R}^{N}}	\phi(|\nabla u_{n}|)\nabla u_{n}\nabla v dx\longrightarrow
\int_{\mathbb{R}^{N}}	\phi(|\nabla u|)\nabla u\nabla v dx.
\end{align*}
Similarly, we have
\begin{align*}
\int_{\mathbb{R}^{N}}	\phi(| u_{n}|) u_{n}v dx\longrightarrow
\int_{\mathbb{R}^{N}}	\phi(|u|)u v dx.
\end{align*}
Therefore,
\begin{align*}
\langle Q'(u_n),v\rangle\longrightarrow \langle Q'(u),v\rangle.
\end{align*} 
By the arbitrariness of $v\in E$, we conclude the results.

\qed

\section{Proof of Theorems \ref{Teo1} and \ref{Teo1.1}}\label{1}
%
%

Initially, note that the condition $(f_4)$ implies that
$\displaystyle\lim_{t\rightarrow+\infty}\frac{f(t)}{\phi_{*}(t)t}=0$. Then, by the conditions $(f_1)$ or $(f_4)$, given $\varepsilon>0 $ there exists $\delta_0>0$, $ \delta_1>0$ and $C_{\varepsilon}>0$ such that
\begin{align}\label{11.16}
K(x)|f(t)|\leq \varepsilon C_1\big(V(x)t\phi(t)+t\phi_{*}(t)\big)+C_{\varepsilon}K(x)t\phi_{*}(t)\chi_{[\delta_0,\delta_1]}(t),
\end{align}
for every $t\geq0$ and $x\in\mathbb{R}^N$, where $C_1=\max\big\{\lVert{K}\lVert_{\infty}, \big\lVert\frac{K}{V}\big\lVert_{\infty}   \big\}$. This inequality yields that the functional $\mathcal{F}:E\longrightarrow\mathbb{R}$ given by
\begin{equation}
\mathcal{F}(u)=\int_{\mathbb{R}^{N}}K(x)F(u)dx
\end{equation}
is well defined and belongs to $ C^1(E,\mathbb R)$ with derivative
$$
\mathcal{F}'(u)v=
\int_{\mathbb R^N}K(x)f(u)vdx, \quad \forall u,v \in E.
$$
From the results presented in the previous section, we can conclude that the energy functional $J:E\longrightarrow\mathbb{R}$ associated with the problem $(P)$, which is given by
$$
J(u)=\int_{\mathbb{R}^{N}} \Phi(|\nabla u|)dx+\int_{\mathbb{R}^{N}}V(x)\Phi(|u|)dx-\int_{\mathbb{R}^{N}}K(x)F(u)dx.
$$
is a continuous and Gateaux-differentiable functional such that $J':E\longrightarrow E^*$ given by
\begin{align*}
J'(u)v=\int_{\mathbb{R}^{N}} \phi(|\nabla u|)\nabla u\nabla v dx+\int_{\mathbb{R}^{N} } V(x)\phi(|u|)uv dx-
\int_{\mathbb R^N}K(x)f(u)vdx
\end{align*}
is continuous from the norm topology of	$E$ to the weak$^*$-topology of $E^*$.

 Once that we intend to find nonnegative solutions for the problem $(P)$, we will assume that
\begin{align}\label{00.22}
f(s)=0,\;\;\;\forall s\in (-\infty,0].
\end{align}

The convexity of the functional $Q$ together with the Gateaux-differentiability of the functional $J$ allows us to present a definition of a critical point for $J$. In this sense, we will say that $u\in E$ is a critical point for the functional $J$ if
\begin{align}\label{0.13}
Q(v)-Q(u)\geq\int_{\mathbb{R}^{N}}{K(x)f(u)(v-u)}dx,\;\;\;\;\forall v\in E.
\end{align}

Our next lemma establishes that a critical point $u$ in the sense \eqref{0.13} is a weak solution for $(P)$.

\begin{proposition} \label{0.4}
	If $u\in E$ is a critical point of $J$ in $E$, then u is a weak solution to $(P)$.
\end{proposition}
\noindent {\it{Proof:}} See Lemma 4.1 in \cite{AlvesandLeandro}.

\qed

Now, let us check that $J$ also satisfies the mountain pass geometry

\begin{lemma}\label{4.9}
	There are $\rho,\eta>0$ such that $J(u)\geq \eta$ if $\lVert u\lVert_{E} = \rho$.
\end{lemma}
\noindent {\it{Proof:}} Consider $0<\varepsilon<\frac{1}{2C_1} $ with $C_1=\big\lVert\frac{K}{V}\big\lVert_{\infty}$. By \eqref{11.16}, there is $C_{\varepsilon}>0$, such that
\begin{align}
K(x)|F(t)|\leq \frac{1}{2}V(x)\Phi(t)+C_{\varepsilon}\Phi_{*}(t),\;\;\forall t\geq 0 \;\text{ and }\; x\in\mathbb{R}^N.
\end{align}
Thus,
\begin{align*}
J(u)&\geq \int_{\mathbb{R}^{N}} \Phi(|\nabla u|)dx+\frac{1}{2}\int_{\mathbb{R}^{N}} V(x)\Phi(|u|)dx-C_{\varepsilon}\int_{\mathbb{R}^{N}}\Phi_{*}(|u|)dx\\
&\geq C_2\big(\xi_{0}(\lVert \nabla u\lVert_{\Phi})+\xi_{0}(\lVert u\lVert_{V,\Phi})\big)-C_2\xi_{3}(\lVert u\lVert_{\Phi_{*}}),
\end{align*}
for some $C_2>0$, where $\xi_{0}(t)=\min\{t^{\ell},\,t^{m}\}$ and $\xi_{3}(t)=\max\{t^{\ell^*},\,t^{m^{*}}\}.$
Choose $\rho>0$ such that
\begin{align*}
\lVert u\lVert_{E}=\lVert u\lVert_{D^{1,\Phi}(\mathbb{R}^{N})}+\lVert u\lVert_{V,\Phi}=\rho<1.
\end{align*}
As $E$ is continuously embedded in $L^{\Phi_{*}}(\mathbb{R}^{N})$, we get $\lVert u\lVert_{\Phi^{*}}\leq1$. Furthermore,
\begin{align*}
J(u)\geq C_2\big(\lVert \nabla u\lVert_{\Phi}^{m}+\lVert u\lVert_{V,\Phi}^{m}\big)-C_2\lVert u\lVert_{\Phi_{*}}^{\ell^*}.
\end{align*}
Using classical inequality
\begin{align*}
(x+y)^{\alpha}\leq 2^{\alpha-1}(x^{\alpha}+y^{\alpha}),\;\;x,y\geq 0\;\; \text{ with }\;\;\;\alpha>1,
\end{align*}
we concluded that
\begin{align*}
J(u)	\geq C_3\lVert u\lVert_{E}^{m}-C_3\lVert u\lVert_{E}^{\ell^*},
\end{align*}
for some positive constant $C_3$. Since $ 0 <m <\ell^*$, there is $\eta> 0 $ such that 
\begin{align*}
J(u)\geq \eta \;\;\;\text{for all}\;\;\lVert u\lVert_{E}=\rho.
\end{align*}
%

\qed

\begin{lemma}\label{4.10}
	There is $e\in E$ with $\lVert e\lVert_{E}>\rho$ and $J(e)<0$.
\end{lemma}
\noindent {\it{Proof:}} Consider $\psi\in C^{\infty}_{0}(\mathbb{R}^{N})\setminus\{0\}$ and $C_{1}\in\mathbb{R}$ such that
\begin{align}\label{4.6}
C_1>\xi_{1}(\lVert \psi\lVert_{D^{1,\Phi}(\mathbb{R}^{N})})+\xi_1(\lVert \psi\lVert_{V,\Phi}).
\end{align}
By $(f_3)$, there exists $C_2>0$ satisfying
\begin{align*}
F(t)\geq C_1|t|^{m}-C_2,\;\;\;\;\forall t\in\mathbb{R}.
\end{align*}
Thus
\begin{align*}
K(x)F(t)\geq C_1 K(x)|t|^{m}-C_2 K(x),\;\;\;\;\forall t\in\mathbb{R}\;\;\text{ and }\;\;x\in\mathbb{R}^{N}.
\end{align*}
That said, considering $t>0$, we have
\begin{align*}
J(t\psi)\leq &\int_{\mathbb{R}^{N}}\Phi(t|\nabla \psi|)dx+\int_{\mathbb{R}^{N}}V(x)\Phi(t|\psi|)dx-C_{1}t^{m}\int_{\mathbb{R}^{N}}K(x)|\psi|^{m}dx+C_{3}|supp(\psi)|,
\end{align*}
before that, it follows from the Lemma \ref{0.3} that
\begin{align*}
J(t\psi)\leq &\xi_{1}(t)\big( \xi_{1}(\lVert \psi\lVert_{D^{1,\Phi}(\mathbb{R}^{N})})+\xi_1(\lVert \psi\lVert_{V,\Phi})\big)-C_{1}t^{m}\int_{\mathbb{R}^{N}}K(x)|\psi|^{m}dx+C_{3}|supp(\psi)|,
\end{align*}
therefore, for $t>1$,
\begin{align}\label{4.7}
\begin{split}
J(t\psi)\leq t^{m}\big( \xi_{1}(\lVert \psi\lVert_{D^{1,\Phi}(\mathbb{R}^{N})})+\xi_1(\lVert \psi\lVert_{V,\Phi})\big)-C_{1}t^{m}\int_{\mathbb{R}^{N}}K(x)|\psi|^{m}dx+C_{3}|supp(\psi)|.
\end{split}
\end{align}
By \eqref{4.6} and \eqref{4.7}, we can conclude that
\begin{align*}
J(t\psi)\longrightarrow-\infty\;\;\text{ as }\;\;t\rightarrow+\infty. 
\end{align*}
The last limit guarantees the existence of $t>0$ large enough that the result is verified with $e = t \psi $. 

\qed

In what follows, let us denote by $c>0$ the mountain pass level associated with $J$, that is,
\begin{align*}
c = \inf_{\gamma \in \Gamma} \max_{t \in [0,1]} J (\gamma(t))
\end{align*}
where
\begin{align*}
\Gamma = \{\gamma \in C([0,1], X) : \; \gamma(0) = 0 \; \text{ and } \; \gamma(1) = e \}.
\end{align*}
Associated with $c$, we have a Cerami sequence $(u_n)\subset E$, that is,
\begin{align}\label{01.39}
J(u_n)\longrightarrow c\;\;\;\;\text{ and }\;\;\;\;
(1+\lVert u_n\lVert)\lVert J'(u_n)\lVert_{*}\longrightarrow 0.
\end{align}
The above sequence is obtained from the Ghoussoub-Preiss theorem, see [\citenum{Motreanu}, Theorem $5.46$].

To show that the sequences obtained in \eqref{01.39}, let us prove a Hardy Type Inequality.

\begin{proposition}[Hardy Type Inequality]\label{22.0}
	Suppose that $(V,K)\in\mathcal{K}_1$, then $E$ is compactly embedded in $L^{Z}_{K}(\mathbb{R}^{N})$, where $ Z(t)= \int_{0}^{|t|}sz(s)ds$ is a $N$-function satisfying
	\begin{equation}\label{22.1}
	0<z_1\leq \dfrac{t^{2}z(t)}{Z(t)}\leq z_2,\;\;\;\;\forall t\geq 0,
	\end{equation}
	where $m<z_1\leq z_2<\ell^{*}$. 
\end{proposition}
\begin{remark}\label{5.01}
	The inequality \eqref{22.1} implies the following inequalities
	$$\xi_{0,Z}(t)Z(\rho)\leq Z(\rho t)\leq\xi_{1,Z}(t)Z(\rho),~~\forall \rho,t\geq0$$ 
	when
	$$\xi_{0,Z}(t)=\min\{t^{z_1},t^{z_2}\}~~\text{and}~~\xi_{1,Z}(t)=\max\{t^{z_1},t^{z_2}\},~~\forall t\geq 0.$$ 
	In the same way of remark \ref{5.02}, we have
	\begin{align*}
	\lim_{t\rightarrow0}\dfrac{Z(|t|)}{\Phi(|t|)}=0\;\;\;\;\text{ e }\;\;\;\;\lim_{t\rightarrow\infty}\dfrac{Z(|t|)}{\Phi_{*}(|t|)}=0.
	\end{align*}
\end{remark}
\noindent {\bf{Proof of Proposition \ref{22.0}:}} We will assume that $(K_2)$ is true. In this case, by Remark \ref{5.01}, given $\varepsilon>0$, there are $0<s_0<s_1$ and $C>0$, such that
\begin{align}\label{5.0}
K(x)Z(|s|)\leq \varepsilon C(V(x)\Phi(|s|)+\Phi_{*}(|s|))+CK(x)\chi_{[s_0, s_1]}(|s|)\Phi_{*}(|s|),
\end{align}
for all $s\in\mathbb{R}$ and $x\in \mathbb{R}^{N}$. Thus, for $r>0$ large enough,
\begin{align}\label{5.1}
\int_{B_r (0)^{c}}K(x)|(|u|)dx\leq \varepsilon CQ(u)+C\Phi_{*}(s_1)\int_{A_u\cap B_r (0)^{c}}K( x)dx,\;\;\;\;\forall u\in E,
\end{align}
where
\begin{align*}
Q(u)=\int_{\mathbb{R}^{N}}V(x)\Phi(|u|)dx+\int_{\mathbb{R}^{N}}\Phi_{*}(| u|)dx
\end{align*}
and
\begin{align*}
A_u=\{x\in\mathbb{R}^{N}:\;s_0\leq |u(x)|\leq s_1\}.
\end{align*}

Consider $ (v_n) $ a bounded sequence in $ E $. To see that the operator $
i: E\longrightarrow L^{B}_{K} (\mathbb{R}^{N})$
is compact just prove that $(v_n)$ has a convergent subsequence on $L^{B}_{K}(\mathbb{R}^{N})$. By Lemma \ref{0.1}, there is $ v \in E $ such that $ v_{n}\xrightharpoonup[\quad]{\ast}v$ in $D ^{1,\Phi}(\mathbb{R}^{N})$, or equivalently $w_{n}\xrightharpoonup[\quad]{\ast}0$ in $D ^{1,\Phi}(\mathbb{R}^{N})$, where $w_n = v_n-v$. Still because of the boundedness $M_1>0$ verifying
\begin{align*}
\int_{\mathbb{R}^{N}} V(x) \Phi(| w_n |) dx \leq M_1 \; \; \; \; \text{ and }\;\;\;\;
\int_{\mathbb{R}^{N}}\Phi_{*}(|w_n|)dx\leq M_1,\;\;\;\;\forall n\in\mathbb{N},
\end{align*}
implying that $(Q(w_n))$ is bounded.

On the other hand, defining
\begin{align*}
A_{n}=\{x\in\mathbb{R}^{N}:\;s_0\leq |w_n(x)|\leq s_1\}
\end{align*}
the last inequality implies that
\begin{align*}
\Phi_{*}(s_0)|A_n|\leq \int_{A_n}\Phi_{*}(|w_n|)dx\leq M_1,\;\;\;\;\forall n\in\mathbb{N }.
\end{align*}
With this, we can guarantee that $\sup_{n\in\mathbb{N}}\,|A_n|<+\infty$. Therefore, by $K_1)$
\begin{align}\label{5.2}
\int_{A_n\cap B_r (0)^{c}}K(x)dx<\dfrac{\varepsilon}{\Phi_{*}(s_1)},\;\;\;\;\forall n\in\mathbb{N}.
\end{align}
From \eqref{5.1} and \eqref{5.2},
\begin{align*}
\int_{B_r (0)^{c}}K(x)Z(|w_n|)dx&\leq \varepsilon C M_1+\Phi_{*}(s_1)\int_{A_n\cap B_r (0)^{c }}K(x)dx\\
&\leq \varepsilon(C M_1+1),\;\;\;\;\forall n\in\mathbb{N},
\end{align*}
and hence,
\begin{align}\label{5.3}
\limsup_{n \to \infty} \int_{B_r (0)^{c}}K(x)Z(|w_n|)dx\leq \varepsilon(C M_1+1).
\end{align}
Since $E$ is compactly embedded in $L^{\Phi}_{loc}(\mathbb{R}^{N})$, so by the Corollary \ref{0.8}, there is a subsequence of $(v_n)$, still denoted by itself, that such
\begin{align*}
v_n\longrightarrow v\;\;\;\;\text{ in }\;\;\;\;L^{\Phi}(B_{r}(0)).
\end{align*}
Thus,  there is a subsequence of $(v_n)$, still denoted by itself, verifying
\begin{align*}
v_n(x)\longrightarrow v(x)\;\;\;\;a.e.\;\text{ in}\;B_{r}(0),
\end{align*}
that is,
\begin{align*}
w_n(x)\longrightarrow 0\;\;\;\;a.e.\;\text{ in }\;B_{r}(0).
\end{align*}
Consider the functions $P:\mathbb{R}\longrightarrow\mathbb{R}$ and $Q:\mathbb{R}\longrightarrow\mathbb{R}$ given by
\begin{align*}
P(t)=Z(|t|)\;\;\;\;\text{ and }\;\;\;\;Q(t)=\Phi_{*}(|t|).
\end{align*}
Of course $ P $ and $ Q $ are continuous in addition
\begin{align*}
\lim_{|t|\rightarrow+\infty}\dfrac{P(t)}{Q(t)}=0.
\end{align*}
Finally, it follows from the boundedness of $ (v_n) $ in $ E $ and from the fact that $ \Phi_{*} \in (\Delta_2) $ the existence of a constant $ C_1> 0 $, so that
\begin{align*}
\int_{\mathbb{R}^{N}} Q(w_n)dx\leq \int_{\mathbb{R}^{N}}\Phi_{*}(|w_n|)dx<C_1,\;\;\;\;\forall n\in\mathbb{N}.
\end{align*}
Invoking the Lemma Strauss
\begin{align*}
\lim_{n \rightarrow \infty}\int_{B_r (0)}P(w_n)dx=0.
\end{align*}
Therefore,
\begin{align}\label{5.4}
\limsup_{n\rightarrow\infty}\int_{B_r (0)}K(x)Z(|w_n|)dx=0.
\end{align} 
According to \eqref{5.3} and \eqref{5.4}, we get
\begin{align*}
\limsup_{n \to \infty} \int_{\mathbb{R}^{N}}K(x)Z(|w_n|)dx \leq&\limsup_{n\rightarrow\infty}\int_{B_r ( 0)}K(x)Z(|w_n|)dx\\
&+\limsup_{n \to \infty} \int_{B_r (0)^{c}}K(x)Z(|w_n|)dx\\
\leq& \varepsilon(C M_1+1).
\end{align*}
By the arbitrariness of $\varepsilon>0$, it follows that
\begin{align*}
\limsup_{n\rightarrow\infty} \int_{\mathbb{R}^{N}}K(x)Z(|w_n|)dx=0.
\end{align*}
As $ Z$ verifies $\Delta_{2}$-condition, we have that
\begin{align*}
w_n\longrightarrow0\;\;\;\;\text{ in }\;\;\;\;L^{Z}_{K}(\mathbb{R}^{N}),
\end{align*}
in other words
\begin{align*}
v_n\longrightarrow v\;\;\;\;\text{ in }\;\;\;\;L^{Z}_{K}(\mathbb{R}^{N}),
\end{align*}
showing the result for the case $(K_2)$.

\qed


Next lemma is an important step to prove that the Cerami sequence obtained in \eqref{01.39} is bounded.

\begin{lemma}\label{6.7}
	Let $(v_n)$ be a bounded sequence in $E$ such that $v_{n}\xrightharpoonup[\quad]{\ast} v$ in $D^{1,\Phi}(\mathbb{R}^{N})$ . Suppose that $f$ satisfies $(f_1)$ or $(f_4)$, then
	\begin{align}\label{44.100}
	\lim_{n\rightarrow\infty}\int_{\mathbb{R}^{N}}K(x)F(v_n)dx=\int_{\mathbb{R}^{N}}K(x)F (v)dx,
	\end{align}
	\begin{align}\label{44.}
	\lim_{n\rightarrow\infty}\int_{\mathbb{R}^{N}}K(x)f(v_n)v_ndx=\int_{\mathbb{R}^{N}}K(x)f (v)vdx
	\end{align}
	and
	\begin{align}\label{44.1,}
	\lim_{n\rightarrow\infty}\int_{\mathbb{R}^{N}}K(x)f(v_n)\psi dx=\int_{\mathbb{R}^{N}}K(x)f (v)\psi dx,\;\;\;\forall \psi \in C^{\infty}_{0}(\mathbb{R}^{N}).
	\end{align}
\end{lemma}

\noindent {\bf{Proof:}} As in \eqref{11.16}, given $\varepsilon>0$, there exists $\delta_0>0$, $ \delta_1>0$ and $C_\varepsilon>0$ such that
\begin{align}\label{6.1}
K(x)f(t)\leq \varepsilon C_1\big(V(x)t\phi(t)+t\phi_{*}(t)\big)+C_\varepsilon K(x)tb(t),\;\;\;\forall t\geq0\text{ and }x\in\mathbb{R}^N
\end{align}
where $C_1=\max\{\lVert{K}\lVert_{\infty}, \big\lVert\frac{K}{V}\big\lVert_{\infty}   \}$. Hence,
\begin{align}\label{6.0}
K(x)F(t)\leq \varepsilon C_{1}(V(x)\Phi(t)+\Phi_{*}(t))+C_\varepsilon B(t),\;\;\; \forall t\geq0\;\text{ and }\;x\in\mathbb{R}^N.
\end{align}
From Proposition \eqref{22.0},
\begin{align}
	\int_{\mathbb {R}^{N}} K(x)B(v_n)dx\rightarrow	\int_{\mathbb {R}^{N}} K(x)B(v)dx
\end{align}
then there is $r_0>0$, so that
\begin{align}\label{22.3}
\int_{B_{{r_{0}}}^c (0)} K(x)B(v_n)dx<\frac{\varepsilon}{C_\varepsilon}, \;\;\;\forall n\in\mathbb{N}.
\end{align}
Moreover, as $(v_{n})$ is bounded in $E$, there is a constant $M_1>0$ satisfying
\begin{align*}
\int_{\mathbb{R}^{N}}V(x)\Phi(|v_{n}|)dx\leq M_1\;\;\;\;\text{ and }\;\;\;\;
\int_{\mathbb{R}^{N}}\Phi_{*}(|v_{n}|)dx\leq M_1,\;\;\;\;\forall n\in\mathbb{N}.
\end{align*}
Combining the last inequalities with \eqref{6.0} and \eqref{22.3},
\begin{align*}
\left|\int_{B_{r_0}^c ( 0)}K(x)F(v_n)dx\right|\leq \varepsilon (C_1 M_1+1),
\end{align*}
for all $n\in\mathbb{N}$. Therefore 
\begin{align}\label{6.2}
\limsup_{n\rightarrow+\infty} \int_{B_{r_0}^c (0)}K(x)F(v_{n})dx\leq \varepsilon (C_1 M_1+1).
\end{align}

On the other hand, using $(f_2)$ and the compactness lemma of Strauss [\citenum{Strauss}, Theorem A.I, p. 338], it follows that
\begin{align}\label{6.3}
\lim_{n\rightarrow+\infty} \int_{B_{r_0} (0)}K(x)F(v_{n})dx=\int_{B_{r_0} (0)}K(x)F(v)dx.
\end{align}
In light of this, we can conclude that
\begin{align*}
\lim_{n\rightarrow+\infty}\int_{\mathbb{R}^{N}}K(x)F(v_{n})dx=\int_{\mathbb{R}^{n}}K(x)F (v)dx.
\end{align*}

To show \eqref{44.}, consider $r_0>0$ given in \eqref{22.3}.
By \eqref{6.1}, 
\begin{align*}
\int_{B_{r_0}^c ( 0)}K(x)f(v_n)v_ndx\leq& \varepsilon C_1\left( m\int_{B_{r_0}^c ( 0)}V(x)\Phi(|v_n|)dx+m^*\int_{B_{r_0}^c ( 0)}\Phi_{*}(|v_n|)dx\right)\\
&+ C_{\varepsilon}z_2\int_{ B_{r_0}^c ( 0)}K(x)Z(v_n)dx\\
\leq& \varepsilon ((m^*+m)C_1M+z_2),
\end{align*}
for all $n\in\mathbb{N}$. Therefore 
\begin{align}\label{44.11}
\limsup_{n\rightarrow+\infty} \int_{B_{r_0}^c (0)}K(x)f(u_{n})u_n dx\leq \varepsilon ((m^*+m)C_1M+z_2),.
\end{align}
On the other hand, using $(f_1)$ or $(f_4)$ and the compactness lemma of Strauss [\citenum{Strauss}, Theorem A.I, p. 338], we can conclude that
\begin{align}\label{44.12}
\lim_{n\rightarrow+\infty} \int_{B_{r_0} (0)}K(x)f(u_{n})u_n dx=0.
\end{align}
Thus, the limit \eqref{44.100} is obtained from \eqref{44.11} and \eqref{44.12}. Related the limit \eqref{44.1,}, it follows directly from the condition $(f_1)$ or $(f_4)$ together with a version of the compactness lemma of Strauss for non-autonomous problem.

\qed

Now, we can prove that the Cerami sequence $(u_n)$ obtained is bounded.

\begin{lemma}\label{5}
	Let $(u_n)$ the Cerami sequence given in \eqref{01.39}. There is a constant $M>0$ such that $J(t u_n)\leq M$ for every $t\in [0,1]$ and $n\in\mathbb{N}$.
\end{lemma}

\noindent{\bf{Proof:}} Let $t_n\in [0,1]$ be such that $J(t_n u_n)=\displaystyle\max_{t\in[0,1]} J(tu_n)$. If $t_n=0$ and $t_n=1$, we are done. Thereby, we can assume $t_n\in(0,1)$ ,and so $J'(t_n u_n)u_n=0$. From this,
\begin{align*}
mJ(t_n u_n)=&mJ(t_n u_n)-J'(t_n u_n)(t_n u_n)\\
=&\int_{\mathbb{R}^{N}}\big(m\Phi(|\nabla(t_nu_n)|)-\phi(|\nabla(t_n u_n)|)|\nabla (t_n u_n) |^{2}\big)dx\\
+\int_{\mathbb{R}^{N}}V(x)&\big(m\Phi(|t_nu_n|)-\phi(|t_n u_n|)|t_n u_n|^{2}\big) dx+\int_{\mathbb{R}^{N}}K(x)\mathcal{H}(t_nu_n)dx,
\end{align*}
where $\mathcal{H}(s)=sf(s)-mF(s)$. By $(f_{2})$ the a function $\mathcal{H}(s)$  is increasing for all $s>0$ and decreasing for $s<0,$ and by $(\phi_4)$ the function $ s\longmapsto m\Phi(s)-\phi(s)s^{2}$ is increasing for $s>0$. Thus,
\begin{align*}
mJ(t u_n)\leq mJ(t_n u_n)\leq mJ(u_n)-J'(u_n)u_n=mJ(u_n)-o_n(1),\;\;\forall t\in [0,1].
\end{align*}
Since $(J(u_n))$ is bounded, there is $M>0$ such that 
\begin{align*}
J(t u_n)\leq M,\;\;\forall t\in [0,1]\;\text{ and } \;n\in\mathbb{N}.
\end{align*}

\qed

\begin{proposition}\label{5.}
	The Cerami sequence $(u_n)$ given in \eqref{01.39} is bounded.
\end{proposition}

\noindent{\bf{Proof:}} Suppose for contradiction that, up to a subsequence, $\lVert u_n\lVert_{E}\longrightarrow\infty$ when $n\rightarrow+\infty$, then we have the following cases:\vspace*{0.1cm}\\
$i)\;\;\lVert \nabla u_n\lVert_{\Phi}\longrightarrow +\infty$ and $(\lVert u_n\lVert_{V,\Phi})$ is bounded\vspace*{0.1cm}\\
$ii)\;\;\lVert u_n\lVert_{V,\Phi}\longrightarrow \infty$ and $(\lVert \nabla u_n\lVert_{\Phi})$ is bounded\vspace*{0.1cm}\\
$iii)\;\;\lVert \nabla u_n\lVert_{\Phi}\longrightarrow +\infty$ and $\lVert u_n\lVert_{V,\Phi}\longrightarrow +\infty$.\vspace*{0.1cm}

In the case $iii) $, consider
\begin{align*}
w_n = \dfrac{ u_n}{\lVert u_n \lVert_{E}}, \; \; \; \; \forall n \in \mathbb{N}.
\end{align*}
Since $ \lVert w_n \lVert_{E} = 1$, by Lemma \ref{0.1}, there exists $w\in E$ such that $w_n\xrightharpoonup[\quad]{\ast}w$ in $D^{1,\Phi}(\mathbb{R}^{N} )$. Now, let us show that $w=0$.
 Before that, as $J(u_n)\longrightarrow c$, we have $J(u_n)\geq 0$, for every $n$ large enough. Thus, there is $n_0\in \mathbb{N}$ such that
\begin{align}\label{6.6}
\int_{\mathbb{R}^{N}}\Phi(|\nabla u_n|)dx+\int_{\mathbb{R}^{N}}V(x)\Phi(|u_n|)dx\geq \int_{\mathbb{R}^{N}} K(x)F(u_n)dx,\;\forall n\geq n_0.
\end{align}
 As $\lVert \nabla u_n\lVert_{\Phi}\geq 1$ and $\lVert u_n\lVert_{V,\Phi}\geq 1$ for every $n\geq n_1$, we have
\begin{align*}
\int_{\mathbb{R}^{N}}\Phi(|\nabla u_n|)dx\leq \lVert \nabla u_n\lVert_{\Phi}^{m}\;\text{ and }\;\int_{\mathbb{R}^{N}}V(x)\Phi(|u_n|) dx\leq \lVert u_n\lVert_{V,\Phi}^{m},\;\forall n\geq n_2.
\end{align*}
So, by \eqref{6.6}
\begin{align*}
\lVert \nabla u_n\lVert_{\Phi}^{m} + \lVert u_n\lVert_{V,\Phi}^{m}\geq \int_{\mathbb{R}^{N}}K(x) F(u_n)dx,\;\;\;\;\forall n\geq n_2.
\end{align*}
Therefore, there is a constant $C>0$ such that
\begin{align*}
C(\lVert \nabla u_n\lVert_{\Phi}+ \lVert u_n\lVert_{V,\Phi})^{m}\geq \int_{\mathbb{R}^{N}}K(x)F (u_n)dx,\;\;\;\;\forall n\geq n_2,
\end{align*}
or equivalently
\begin{align*}
C(\lVert u_n\lVert_{E})^{m}\geq \int_{\mathbb{R}^{N}}K(x)F(u_n)dx,\;\;\;\;\forall n\geq n_2
\end{align*}
where $n_2=\max\{n_0,n_1\}$. Thus,
\begin{align*}
C\geq \int_{\mathbb{R}^{N}} K(x)\dfrac{F(u_n)}{\lVert u_n\lVert_{E}^{m}}dx\geq \int_{\mathbb {R}^{N}} K(x)\dfrac{F(u_n)}{\lvert u_n\lvert^{m}}|w_n|^{m}dx.
\end{align*}
The condition $(f_4)$ implies that for every $\tau >0$, there is $\xi>0$ sufficiently large such that
\begin{align*}
\dfrac{F(s)}{|s|^{m}}\geq \tau,\;\;\;\;\forall\;|s|\geq \xi.
\end{align*}
So
\begin{align*}
1+C\geq\int_{\Omega\cap\{|u_n|\geq \xi\}} K(x)\dfrac{F(u_n)}{\lvert u_n\lvert^{m}}|w_n| ^{m}dx\geq \tau \int_{\Omega\cap\{|u_n|\geq \xi\}} K(x)|w_n|^{m}dx,
\end{align*}
where $\Omega=\{x\in\mathbb{R}^{N}:\;w(x)\neq 0\}$. By Fatou's Lemma,
\begin{align*}
1+C\geq\tau\int_{\Omega} K(x)|w(x)|^{m}dx,\;\;\;\;\forall\;\tau>0.
\end{align*}
Therefore,
\begin{align*}
\int_{\Omega} K(x)|w(x)|^{m}dx=0.
\end{align*}
As $K(x)>0$ for almost everywhere in $\mathbb{R}^{N}$, we have $w=0$.

Note that for every $M>1$, there is $n_0\in \mathbb{N}$ such that, $\dfrac{M}{\lVert \nabla u_n\lVert_{\Phi}}\in[0,1]$, for all $n\geq n_0.
$
Given this, we get
\begin{align*}
J(t_n u_n)\geq& J( \dfrac{M}{\lVert \nabla u_n\lVert_{\Phi}} u_n)= J( \dfrac{M}{\lVert \nabla u_n\lVert_{\Phi}} |u_n|)=J(Mw_n)\\
=&\int_{\mathbb{R}^{N}}\Phi(M|\nabla w_n|)dx+\int_{\mathbb{R}^{N}}V(x)\Phi(M|w_n| )dx-\int_{\mathbb{R}^{N}}K(x)F(Mw_n)dx\\
\geq&MQ(u_n)-\int_{\mathbb{R}^{N}}K(x)F(Mw_n)dx.
\end{align*}
By definition of the sequence $(w_n)$, we have
$
\lVert \nabla w_n\lVert_{\Phi}\leq 1$ and $
\lVert w_n\lVert_{V,\Phi}\leq 1,$ for all $n\in\mathbb{N}.
$
Then,
\begin{align*}
\int_{\mathbb{R}^{N}}\Phi(|\nabla w_n|)dx\geq \lVert \nabla w_n\lVert_{\Phi}^{m},\text{ and }
\int_{\mathbb{R}^{N}}V(x)\Phi(|w_n|)dx\geq \lVert w_n\lVert_{V,\Phi}^{m},\;\forall\;n\in\mathbb{N}.
\end{align*}
So there is $C>0$ such that
\begin{align*}
Q(w_n)\geq \lVert \nabla w_n\lVert_{\Phi}^{m}+\lVert w_n\lVert_{V,\Phi}^{m}\geq C(\lVert \nabla w_n\lVert_{\Phi}+\lVert w_n\lVert_{V,\Phi})^{m},\;\;\;\;\forall\;n\in\mathbb{N}.
\end{align*}
Thus
\begin{align*}
J(t_n u_n)&\geq M C(\lVert w_n\lVert_{E})^{m}-\int_{\mathbb{R}^{N}}K(x)F(Mw_n)dx\\
&= M C- \int_{\mathbb{R}^{N}}K(x)F(Mw_n)dx.
\end{align*}
By Lemma \ref{6.7}, 
\begin{align*}
\lim_{n \rightarrow \infty} \int_{\mathbb{R}^{N}}K(x)F(Mw_n)dx=0,
\end{align*}
therefore,
\begin{align*}
\liminf_{n \rightarrow \infty} J(t_n u_n) \geq M, \; \; \; \; \forall \; M \geq 1.
\end{align*}
which constitutes a contradiction with Lemma \ref{5}, once that $(J(t_n u_n))$ is bounded from above. Therefore $(u_n)$ is bounded.

The cases $i)$ and $ii)$ are analogous to the case $iii)$.

\qed

Since that the Cerami sequence $(u_n)$ given in \eqref{2.1} is bounded in  $E$, by Lemma \ref{0.6}, we can assume that for some subsequence, there is $u \in   E$ such that
\begin{align}\label{6.4}
u_{n}\xrightharpoonup[\quad]{\ast} u\;\;\text{ in }D^{1,\Phi}(\mathbb{R}^{N}),\;\;\;\text{ and }\;\;\;
u_{n}(x)\longrightarrow u(x)\;\;a.e.\;\;\mathbb{R}^{N}.
\end{align}

\begin{lemma}\label{0.2}
	The Cerami sequence $(u_n)$ given in \eqref{2.1} satisfies $$ \int_{\mathbb{R}^{N}} \Phi(|\nabla u|)dx\leq\displaystyle \liminf_{n\rightarrow\infty }\int_{\mathbb{R}^{N}} \Phi(|\nabla u_n|)dx.$$
\end{lemma}
\noindent{\it{Proof:}}
Consider $\varphi\in L^{\infty}(\mathbb{R}^{N})$ arbitrary. For every $R>1$, define the function
\begin{align*}
\omega_R(t)=\left\{\begin{array}{lr}
1\;,& \text{se } \;x\in B_R(0)\\
0\;,&\text{ se } x\in B_R^c(0)
\end{array}\right..
\end{align*}
It is clear that $\omega_R\in E^{\tilde{\Phi}}(\mathbb{R}^{N})$, because $\omega_R\in L^{\infty }(\mathbb{R}^{N})$ and $ supp(\omega_R)\subset\subset \mathbb{R}^{N}$. As a consequence, we have $\varphi\omega_R\in E^{\tilde{\Phi}}(\mathbb{R}^{N})$. We know that $ L^{\Phi}(\mathbb{R}^{N})\xhookrightarrow[cont\;]{} L^{1}_{loc}(\mathbb{R}^{N})$, such as $u_n, u\in D^{1,\Phi}(\mathbb{R}^{N}),$ then
\begin{align*}
\dfrac{\partial u_n}{\partial x_i},\;\dfrac{\partial u}{\partial x_i}\in L^{1}_{loc}(\mathbb{R}^{N}), \;\;\;\;\forall i=1,2,\cdots,N
\end{align*}
and hence
\begin{align*}
\dfrac{\partial u_n}{\partial x_i}\omega_R,\;\dfrac{\partial u}{\partial x_i}\omega_R\in L^{1}_{loc}(\mathbb{R}^{ N}),\;\;\;\;\forall i=1,2,\cdots,N.
\end{align*}
By \eqref{6.4}, $u_n \xrightharpoonup[\quad]{\ast} u$ in $D^{1,\Phi}(\mathbb{R}^{N})$, thus
{\begin{align*}
	\int_{\mathbb{R}^{N}}\left(\dfrac{\partial u_n}{\partial x_i}\omega_R\right)\varphi dx=\int_{\mathbb{R}^{N}}\dfrac {\partial u_n}{\partial x_i}(\omega_R \varphi)dx\rightarrow\int_{\mathbb{R}^{N}}\dfrac{\partial u}{\partial x_i}(\omega_R \varphi)dx= \int_{\mathbb{R}^{N}}\left(\dfrac{\partial u}{\partial x_i}\omega_R\right )\varphi dx.
	\end{align*}}
By the arbitrariness of $\varphi$ in $L^{\infty}(\mathbb{R}^{N})$, 
\begin{align*}
\dfrac{\partial u_n}{\partial x_i}\omega_R \xrightharpoonup[\quad]{} \dfrac{\partial u}{\partial x_i}\omega_R,\;\;\text{ in }\;\; L^{1}(\mathbb{R}^{N})
\end{align*}
because $(L^{1}(\mathbb{R}^{N}))^*=L^{\infty}(\mathbb{R}^{N})$. Therefore, applying [\citenum{ET}, Theorem $2.1$], we can conclude that
\begin{align*}
\int_{B_R(0)}\Phi(|\nabla u|)dx\leq \liminf_{n\rightarrow\infty} \int_{\mathbb{R}^{N}} \Phi(|\nabla u_n| \omega_R)dx\leq \liminf_{n\rightarrow\infty} \int_{\mathbb{R}^{N}} \Phi(|\nabla u_n|)dx.
\end{align*}
Passing the limit at $R\rightarrow+\infty$, we get
\begin{align*}
\int_{\mathbb{R}^{N}}\Phi(|\nabla u|)dx\leq \liminf_{n\rightarrow\infty} \int_{\mathbb{R}^{N}} \Phi( |\nabla u_n|)dx,
\end{align*}
which completes the proof.

\qed

Fix $v\in C^{\infty}_0(\mathbb{R}^N)$. By boundedness of Cerami sequence $(u_n)$, we have $J'(u_n)(v-u_n)=o_n(1)$, hence, since $\Phi$ is a convex function, it is possible to show that
\begin{align}\label{7.2}
\begin{split}
Q(v)-Q(u_n)\geq \int_{\mathbb{R}^{N}}K(x)f(u_n)(v-u_n)dx+o_n(1).
\end{split}
\end{align}
For the Lemma \ref{0.2}, we have to
\begin{align*}
\liminf_{n\rightarrow\infty}\int_{\mathbb{R}^{N}}\Phi(|\nabla u_n|)dx\geq\int_{\mathbb{R}^{N}}\Phi( |\nabla u|)dx.
\end{align*}
Now, due to the Fatou's Lemma, we have
\begin{align*}
\liminf_{n\rightarrow\infty}\int_{\mathbb{R}^{N}}V(x)\Phi(| u_n|)dx\geq\int_{\mathbb{R}^{N}}V (x)\Phi(|u|)dx.
\end{align*}
Therefore,
\begin{align}\label{7.3}
\liminf_{n\rightarrow\infty} Q(u_n)\geq Q(u).
\end{align}
From \eqref{7.2} and \eqref{7.3} together with the Lemma \ref{6.7}, we get
\begin{align*}
Q(v)-Q(u)\geq\int_{\mathbb{R}^{N}}K(x)f(u)(v-u)dx,\;\;\;\;\forall\;v \in C^{\infty}_{0}(\mathbb{R}^{N}).
\end{align*}
As $E=\overline{C^{\infty}_{0}(\mathbb{R}^{N})}^{\lVert\cdot\lVert_{E}}$ and $\Phi\in( \Delta_{2})$, we conclude that
\begin{align}\label{7.4}
Q(v)-Q(u)\geq\int_{\mathbb{R}^{N}}K(x)f(u)(v-u)dx,\;\;\;\;\forall\;v \in E.
\end{align}
In other words, $u$ is a critical point of the $J$ functional. And from Proposition  \ref{0.4}, we can conclude that $u$ is a weak solution for $(P)$. Now, we substitute $v=u^+:=\max\{0,u(x)\}$ in \eqref{7.4} and use \eqref{00.22} to get
\begin{align*}
-\int_{\mathbb{R}^{N}}\Phi(|\nabla u^-|)dx-\int_{\mathbb{R}^{N}}V(x)\Phi(u^-)dx\geq\int_{\mathbb{R}^{N}}K(x)f(u)u^-dx=0,
\end{align*}
which leads to
\begin{align*}
\int_{\mathbb{R}^{N}}\Phi(|\nabla u^-|)dx=0\;\;\text{ and }\;\;\int_{\mathbb{R}^{N}}V(x)\Phi(u^-)dx=0
\end{align*}
whence it is readily inferred that $u^-=0$, therefore, $u$ is a weak nonnegative solution.

Note that $u$ is nontrivial. In the sense, consider a sequence $(\varphi_k)\subset C^{\infty}_{0}(\mathbb{R}^{N})$ such that $\varphi_k \rightarrow u$ in $D^{1 ,\Phi}(\mathbb{R}^{N})$. Since $(u_n)$ bounded, we get $J'(u_n)(\varphi_k-u_n)=o_n(1)\lVert \varphi_k\lVert-o_n(1)$. As $\Phi$ is convex, we can show that
\begin{align}\label{7.11}
Q (\varphi_k) -Q (u_n) \geq \int_{\mathbb{R}^{N}} K(x) f(u_n)(\varphi_k-u_n)dx + o_n(1)\lVert \varphi_k\lVert-o_n(1).
\end{align}
Since $(\lVert \varphi_k\lVert)_{k\in\mathbb{N}}$ is a bounded sequence, it follows from \eqref{7.11} and from Lemma \eqref{6.7} that
\begin{align*}
Q(\varphi_k)\geq \limsup_{n \to \infty}Q(u_n)+\int_{\mathbb{R}^{N}} K(x) f(u)(\varphi_k-u)dx.
\end{align*}
Now, note that being $\Phi\in (\Delta_{2})$ and $\varphi_k \rightarrow u$ in $E$, we conclude from the inequality above that
\begin{align} \label{7.19}
Q(u) \geq \limsup_{n \rightarrow \infty} Q(u_n).
\end{align}
From \eqref{7.3} and \eqref{7.19},
\begin{align}\label{7.20}
\lim_ {n \rightarrow \infty} Q(u_n)=Q(u).
\end{align}
By Lemma \ref{6.7}, we have 
\begin{align}\label{7.21}
\lim_{n\rightarrow\infty} \int_{\mathbb{R}^{N}}K(x)F(u_n)dx=\int_{\mathbb{R}^{N}}K(x)F(u)dx,
\end{align}
Therefore, by \eqref{01.39}, \eqref{7.20} and \eqref{7.21}, we conclude
\begin{align*}
0<c=\lim_{n\rightarrow\infty} J(u_n)=\int_{\mathbb{R}^{N}} \Phi(|\nabla u|)dx-\int_{\mathbb{R}^{N}} K(x)F(u)dx=J(u),
\end{align*}
that is, $u\neq 0$.

\subsection{\textbf{Proof of Theorem \ref{Teo1}}}\label{sec1} Assuming the assumptions of Theorem \ref{Teo1}, the above argument guarantees the existence of a nontrivial weak solution for problem $(P)$. This subsection makes heavy use of fact that $\Phi\in \mathcal{C}_m$ to study the regularity of nonnegative solutions of the problem $(P)$, thereby showing Theorem \ref{Teo1}.

We begin by presenting a technical result, which can be found in \cite{Ailton}.
\begin{lemma}\label{22.6}
	Let  $u\in E$ be a nonnegative solution of $(P)$, $x_0 \in  \mathbb{R}^N$ and $R_0>0$. Then 
	\begin{align*}
\int_{\mathcal{A}_{k,t}} |\nabla u|^m dx\leq C \left( \int_{\mathcal{A}_{k,s}}\left|\dfrac{u-k}{s-t}\right|^{m^*}dx +(k^{m^*}+1) |\mathcal{A}_{k,s}|\right)
\end{align*}
	where $0<t<s<R_0$, $k>1$, $ \mathcal{A}_{k,\rho}=\{x\in B_{\rho}(x_0): u(x)>k\}$ and $C>0$ is a constant that does not depend on $k$.
\end{lemma}

\noindent{\bf{Proof:}}
Let $u\in E$ be a weak solution nonnegative of $(P)$ and $x_0\in \mathbb{R}^{N}$. Moreover, fix $0 < t < s < R_0 $ and $\zeta \in C^{\infty}_0(\mathbb{R}^N)$ verifying 
\begin{align*}
	0\leq \zeta \leq 1, \quad supp (\zeta)\subset B_{s}(x_0),\quad \zeta\equiv 1\;\text{ on }\; B_t(x_0)\;\;\text{ and }\;\; |\nabla \zeta|\leq \dfrac{2}{s-t}.
\end{align*}
For $k >1$, set $\varphi = \zeta^m(u-k)^{+}$ and
\begin{align*}
J=\int_{\mathcal{A}_{k,s}}\Phi(|\nabla u|)\zeta ^mdx,
\end{align*}
Using $\varphi$ as a test function and $\ell \Phi(t)\leq \phi(t)t^2$, we find
%
\begin{align*}
\ell J\leq m \int_{\mathcal{A}_{k,s}}\zeta^{m-1}(u-k)^+ \phi(|\nabla u|)|\nabla u| |\nabla\zeta| dx
-\int_{\mathcal{A}_{k,s}}V(x) \phi( u) u\zeta^m (u-k)^+ dx\\+ \int_{\mathcal{A}_{k,s}}K(x) f(u) \zeta^m (u-k)^+ dx.
\end{align*}
As in \eqref{11.16}, given $\eta > 0$, there exists $C_\varepsilon>0$ such that
\begin{align*}
K(x)f(t)\leq \eta C_1V(x)t\phi(t)+C_\varepsilon t\phi_{*}(t), \;\;\;\forall t\geq0\text{ and } x\in\mathbb{R}^N.
\end{align*}
where $C_1=\lVert{K}\lVert_{\infty}$. Thus, setting $\eta = \frac{1}{C_1}$, there exists a constant $C_2>0$ such that
\begin{align}\label{11.2}
\ell J\leq m \int_{\mathcal{A}_{k,s}}\zeta^{m-1}(u-k)^+ \phi(|\nabla u|)|\nabla u| |\nabla\zeta| dx
+C_2\int_{\mathcal{A}_{k,s}} \phi_*(| u|) u\zeta^m (u-k)^+ dx.
\end{align}
For each $\tau \in(0, 1)$, the Young’s inequalities gives
%
%
\begin{align}\label{11.0}
\begin{split}
\phi(|\nabla u|)|\nabla u||\nabla \zeta| \zeta^{m-1}(u-k)^+\leq \tilde{\Phi }( 	\phi(|\nabla u|)|\nabla u|\zeta^{m-1}\tau)+C_3\Phi\Big(\Big|\dfrac{u-k}{s-t}\Big|\Big).
\end{split}
\end{align}
It follows from Lemma \ref{0.51}, 
\begin{align}\label{11.1}
\tilde{\Phi }( 	\phi(|\nabla u|)|\nabla u|\zeta^{m-1}\tau)\leq C_4(\tau \zeta^{m-1})^\frac{m}{m-1}{\Phi }(	|\nabla u|).
\end{align}	
From \eqref{11.2}, \eqref{11.0} and \eqref{11.1}, 
\begin{align*}
\ell J\leq m  C_4\tau^\frac{m}{m-1}\int_{\mathcal{A}_{k,s}} \Phi (	|\nabla u|) \zeta^{m}+mC_3\int_{\mathcal{A}_{k,s}} \Phi\Big(\Big|\dfrac{u-k}{s-t}\Big|\Big)dx+C_2\int_{\mathcal{A}_{k,s}} u\phi_*( u) \zeta^m (u-k)^+ dx.
\end{align*}
Choosing $\tau\in (0,1)$  such that $0<m  C_4\tau^\frac{m}{m-1}<\ell$, we derive
\begin{align}\label{11.6}
J\leq C_5\int_{\mathcal{A}_{k,s}} \Phi\Big(\Big|\dfrac{u-k}{s-t}\Big|\Big)dx+C_5\int_{\mathcal{A}_{k,s}} u\phi_*(u) \zeta^m (u-k)^+ dx.
\end{align}
By Young’s inequalities,
\begin{align}\label{11.7}
u\phi_*( u)\zeta^m (u-k)^+ \leq C_6\Phi_*\left(\Big|\dfrac{u-k}{s-t}\Big| \right)+ C_6\Phi_*(k).
\end{align}
Therefore, a combination of \eqref{11.6} and \eqref{11.7}, yields 
\begin{align}\label{111}
J\leq C_{7}\int_{\mathcal{A}_{k,s}} \Phi\Big(\Big|\dfrac{u-k}{s-t}\Big|\Big)dx+C_{7}\int_{\mathcal{A}_{k,s}} \Phi_*\Big(\Big|\dfrac{u-k}{s-t}\Big|\Big)dx+ C_{7} \int_{\mathcal{A}_{k,s}}\Phi_*(k)dx.
\end{align}
Now, using that $\ell\leq m<m^* $ and applying the Lemmas \ref{0.3} and \ref{0.5} for functions $\Phi$ and $\Phi_{*},$ respectivamente, we get 
\begin{align*}
	\Phi\Big(\Big|\dfrac{u-k}{s-t}\Big|\Big)\leq \Phi(1)\Big( \Big|\dfrac{u-k}{s-t}\Big|^{m^*}+1 \Big),\;\Phi_*\Big(\Big|\dfrac{u-k}{s-t}\Big|\Big)\leq \Phi_*(1)\Big( \Big|\dfrac{u-k}{s-t}\Big|^{m^*}+1 \Big)\text{ and }\Phi_*(k)\leq(k^{m^*}+1).
\end{align*}
From \eqref{111} and the inequality above,
\begin{align*}
	J\leq C_{8}\left(\int_{\mathcal{A}_{k,s}} \Big|\dfrac{u-k}{s-t}\Big|^{m^*}dx + (k^{m^*}+1) |\mathcal{A}_{k,s}|\right)
\end{align*}
The result follows from the fact that $\Phi\in \mathcal{C}_m$.

\qed
%
%
%
%
\begin{lemma}\label{22.21}
Let $u \in E$ be a nonnegative solution of $(P)$. Then, $u \in L^\infty_{loc} (\mathbb{R}^N)$.
\end{lemma}
\noindent{\bf{Proof:}} To begin with, consider $\Lambda$ a compact subset on $\mathbb{R}^N$. Fix $R_1\in (0,1)$ and $x_0\in\Lambda$. Given $K>1$, define the sequences
\begin{align*}
\sigma_n=\dfrac{R_1}{2}+\dfrac{R_1}{2^{n+1}},\;\;\overline{\sigma}_n=\dfrac{\sigma_n +\sigma_{n+1}}{2}\;\;\text{ and }\;\; K_n=\dfrac{K}{2}\left(1-\dfrac{1}{2^{n+1}}\right).
\end{align*}
%
%
%
%
%
For every $n\in\mathbb{N}$, we consider
\begin{align*}
J_n=\int_{\mathcal{A}_{K_n,\sigma_n}}\big((u-K_n)^+\big)^{m^*}dx\;\;\text{ and }\;\;
\xi_n=\xi\left(\dfrac{2^{n+1}}{R_1}\Big(|x-x_0|-\dfrac{R_1}{2}\Big)\right),\;\;x\in\mathbb{R}^N
\end{align*}
with $\xi\in C^{1}(\mathbb{R})$ satisfying 
\begin{align*}
0\leq \xi\leq 1,\quad \xi(t)=1\;\text{ for }\; t\leq \dfrac{1}{2}\quad\text{ and }\quad  \xi(t)=0\;\text{ for }\; t\geq \dfrac{3}{4}..
\end{align*}
Repeating the same arguments used in the proof of Lemma 3.6 in \cite{Ailton},  there exists a constant $C=C(N,m,R_1)>0$ such that
\begin{align}\label{22.25}
\begin{split}
J_{n+1}\leq CD^n J_n^{1+\omega},\;\;n=0,1,2,\cdots
\end{split}
\end{align}
where $D=2^{(m+m^*){\frac{m^{*}}{m}}}$ and $\omega=\frac{m^*}{m}-1$.
Note that 
\begin{align}\label{22.16}
\begin{split}
J_0&=\int_{\mathcal{A}_{K_{0},{\sigma_0}}}\big((u-K_{0})^+\big)^{m^*}dx\leq \int_{B_{R_1}(x_0)}\big((u-K_{0})^+\big)^{m^*}dx
\end{split}
\end{align}
%
Then, by the Lebesgue’s Theorem, $\displaystyle\lim_{K\rightarrow\infty} J_0 =0$, from where it follows
that
\begin{align*}
J_0\leq C^{-\frac{1}{\omega}} D^{-\frac{1}{\omega^2}} , \;\;\text{ for all }\; K\geq K^*
\end{align*}
for some $K^*\geq 1$ that depends on $x_0$. Fix $K = K^*$. Thus, by [\citen{O.A}, Lemma 4.7], we deduce that
\begin{align*}
J_n \longrightarrow 0\;\;\text{ as}\;\;n\rightarrow \infty.
\end{align*}
On the other hand, 
\begin{align*}
\lim_{n\rightarrow\infty}J_n=\lim_{n\rightarrow\infty}\int_{\mathcal{A}_{K_n,\sigma_n}}\big((u-K_{0})^+\big)^{m^*}dx=\int_{\mathcal{A}_{\frac{K^*}{2},\frac{R_1}{2}}}\big((u-\frac{K^*}{2})^+\big)^{m^*}dx
\end{align*}
Hence,
\begin{align*}
\int_{\mathcal{A}_{\frac{K^*}{2},\frac{R_1}{2}}}\big((u-\frac{K^*}{2})^+\big)^{m^*}dx=0,
\end{align*}
leading to
\begin{align*}
u(x)\leq \frac{K^*}{2},\;\;\text{ a.e. }\text{in } B_{\frac{R_1}{2}}(x_0).
\end{align*}
Since $x_0$ is arbitrary and $\Lambda$ is a compact subset, the last inequality ensures that
\begin{align*}
u(x)\leq \frac{M}{2},\;\;\text{ a.e. }\text{in } \Lambda
\end{align*}
for some constant $M>0$. By the arbitrariness of $\Lambda$, we conclude that $u\in L^{\infty}_{loc}(\mathbb{R}^{N})$. 

\qed

This Lemmas guarantees that the Theorem \ref{Teo1} is valid. The proof of Theorem \ref{Teo11} will be divided into the following Lemmas:

\begin{lemma}
	$u\in C^{1,\alpha}_{loc} (\mathbb{ R}^N)$.
\end{lemma}

\begin{proof}
It is enough to apply the regularity theorem due to Lieberman [\citen{O.A}, Theorem 1.7]. And this is possible due to the condition $(\phi_{6})$.
\end{proof}

\begin{corollary}\label{22.2}
	 Let $u\in E$ be a nonnegative solution of $(P)$. Then, $u$ is positive solution.
\end{corollary}
\noindent{\bf{Proof:}}  If  $\Omega\subset \mathbb{R}^N$ is a bounded domain, the Lemma \ref{22.21} implies that $u\in C^{1}(\overline{\Omega})$.
Using this fact, in
the sequel, we fix  $M_1>\max\big\{\lVert \nabla u\lVert_{L^{\infty}(\overline{\Omega})}, 1\big\}$ and
\begin{align*}
\varphi(t)=\left\{\begin{array}{lr}
\;\;\; \;\phi(t)\;\;\;, \;\;\;\;\;\;\text{ for } 0< t\leq M_1\;\;\;\vspace*{0.2cm}\\
\dfrac{\phi(M_1)}{M_1^{\beta-2}}t^{\beta-2}\;,\text{ for } t\geq M_1
\end{array}\right.,
\end{align*}
where $\beta$ is given in the hypothesis $(\phi_{5})$. Still by the condition $(\phi_5)$, there are $\alpha_1, \alpha_2>0$ satisfying
\begin{align}\label{2.23}
\varphi(|y|)|y|^2=	\phi(|y|)|y|^2\geq\alpha_1|y|^{\beta} \;\;\;\text{ and }\;\;\;|\varphi(|y|)y|\leq \alpha_2 |y|^{\beta-1},\;\;\;\forall y \in\mathbb{R}^N.
\end{align}
Now, consider the vector measurable functions $G:\Omega\times \mathbb{R} \times \mathbb{R}^N \longrightarrow \mathbb{R}^N$ given by $\linebreak G(x,t,p)=\frac{1}{\alpha_1}\varphi (|p|)p$. From \eqref{2.23},
\begin{align}
|G(x,t,p)|\leq\frac{\alpha_2}{\alpha_1}|p|^{\beta-1} \;\;\;\text{ and }\;\;\;pG(x,t,p)\geq |p|^{\beta-1} , \;\;\text{ for all }\; (x,t,p)\in \Omega\times \mathbb{R} \times \mathbb{R}^N.
\end{align}

We next will  consider the scalar measurable function $B:\Omega\times \mathbb{R} \times \mathbb{R}^N \longrightarrow \mathbb{R}$ given by $G(x,t,p)=\frac{1}{\alpha_1}\big(V(x)\phi (|t|)t-K(x)f(t)\big)$. Remember that from the inequality \eqref{11.16}, there will be a constant $C_1>0$ satisfying
\begin{align}
	K(x)|f(t)|\leq C_1V(x)\phi(|t|)|t|+ C_1\phi_*(|t|)|t|, \;\;\;\forall t\in\mathbb{R}^N\;\text{ and }\;x\in\mathbb{R}^N.
\end{align}

Fix $M\in(0,\infty)$. Through the condition $(\phi_{5})$ and by a simple computation yields there exists $C_2=C_2(M) > 0$ verifying
%
\begin{align*}
|B(x,t,p)|\leq C_2|t|^{\beta-1},\; \text{ for every }\; (x,t,p)\in \Omega\times (-M,M) \times \mathbb{R}^N  . 
\end{align*}
By the arbitrariness of $M$, we can conclude that functions $G$ and $B$ fulfill the structure required by Trudinger \cite{Trudinger}. Also, as $u$ is a weak solution of $(P)$, we infer that $u$ is a quasilinear problem solution
\begin{align*}
	-div\, G(x,u,\nabla u(x))+B(x,u,\nabla u(x))=0\;\text{ in }\Omega.
\end{align*}
By [\citenum{Trudinger}, Theorem 1.1], we deduce that $u>0$ in $\Omega$. As $
\Omega$ arbitrary, we conclude that $u>0$ in $\mathbb{R}^{N}$.

\qed


\subsection{\textbf{Proof of Theorem \ref{Teo1.1}}}
%
Assuming the assumptions of Theorem \ref{Teo1.1}, the arguments at the beginning of this section 4, guarantees the existence of a nontrivial weak solution for problem $(P)$. Then, to prove the Theorem \ref{Teo1.1}, it is enough to show the regularity of nonnegative solutions of the problem $(P)$. Following the same arguments as the proof of Lemma \ref{22.6}, we state the following result.
\begin{lemma}\label{22.00}
	Let  $u\in E$ be a nonnegative solution of $(P)$, $x_0 \in  \mathbb{R}^N$ and $R_0>0$. Then 
	\begin{align*}
	\int_{\mathcal{A}_{k,t}} |\nabla u|^\ell dx\leq C \left( \int_{\mathcal{A}_{k,s}}\left|\dfrac{u-k}{s-t}\right|^{\ell^*}dx +(k^{\ell^*}+1) |\mathcal{A}_{k,s}|\right)
	\end{align*}
	where $0<t<s<R_0$, $k>1$, $ \mathcal{A}_{k,\rho}=\{x\in B_{\rho}(x_0): u(x)>k\}$ and $C>0$ is a constant that does not depend on $k$.
\end{lemma}

The following result is proved in an entirely analogous way to Lemma \ref{22.21}, it is enough to change $m$ for $\ell$ in every proof.

\begin{lemma}
	Let $u \in E$ be a nonnegative solution of $(P)$. Then, $u \in L^\infty_{loc} (\mathbb{R}^N)$.
\end{lemma}

Arguing as in the proof of Corollary \ref{22.2}, the nonnegative solution of $(P)$ is positive solution. Showing the Theorem \ref{Teo1.1}.

\begin{remark}
In the Theorems \ref{Teo1} and \ref{Teo1.1}, assuming that $\displaystyle\inf_{x\in\mathbb{R }^N}V(x)>0$, the nonnegatives solution for $(P )$ and they go to zero at infinity, that is, $\displaystyle\lim_{|x|\rightarrow\infty} u(x)=0$.
\end{remark}
%

\section{Proof of Theorems \ref{Teo3} and \ref{Teo3.1}}\label{21}
In the beginning, we are going to present the ardy Type Inequality for case $(K,V)\in \mathcal{K}_2 $. 

	\begin{remark}\label{5.02}
	The inequality \eqref{5.04} implies the following inequalities
	$$\xi_{0,A}(t)A(\rho)\leq A(\rho t)\leq\xi_{1,A}(t)A(\rho),~~\forall \rho,t\geq0$$
	when
	$$\xi_{0,A}(t)=\min\{t^{a_1},t^{a_2}\}~~\text{and }~~\xi_{1,A}(t)=\max\{t^{a_1},t^{a_2}\},~~\forall t\geq 0.$$ Besides by Lemma \ref{0.3} and Lemma \ref{0.5}, we have
	\begin{align*}
	\lim_{t\rightarrow0} \dfrac{A(t)}{\Phi(t)}=0\;\;\text{ and }\;\;\lim_{|t|\rightarrow\infty} \dfrac{A(t)}{\Phi_{*}(t)}=0
	\end{align*}
\end{remark}

	\begin{proposition}[Hardy Type Inequality]\label{44.01}
		If $(V,K)\in\mathcal{K}_2$, then $E$ is compactly embedded in$L^{A}_{K}(\mathbb{R}^{N})$ where $A$ is given in $K_3$.
	\end{proposition}

	\noindent {\bf{Proof:}}  As $E$ is continuously embedded in $L^{\Phi_*}(\mathbb{R}^N)$, there exists $C_1>0$ such that 
		\begin{align}\label{2}
			\lVert u\lVert_{\Phi_{*}}\leq C_1\lVert u\lVert_{E},\;\;\forall u\in E.
		\end{align}
	 By $ (K_3) $ given $ \varepsilon> 0 $ there is $ r> 0 $ large enough such that
	\begin{align}\label{01.9}
	K(x)A(s)\leq \varepsilon(V(x)\Phi(|s|)+\Phi_{*}(|s|)),
	\end{align}
	for all $s\in\mathbb{R}$ and $|x|\geq r$. On the other hand, by the Remark \ref{5.02}, there is a constant $C_2>0$ such that
	\begin{align*}
	A(t)\leq C_2\Phi(t)+C_2\Phi_{*}(t),\;\;\;\; \forall t>0.
	\end{align*}
	Hence, for each $x\in B_{r}(0)$,
	\begin{align}\label{01.8}
	\begin{split}
	K(x)A(t)\leq C_2\left\lVert\frac{K}{V}\right\lVert_{L^\infty( B_{r}(0))}V(x)\Phi(t)+C_2\lVert K\lVert_{\infty}\Phi_{*}(t),\;\;\;\;\forall t>0.
	\end{split}
	\end{align}
	Combining \eqref{01.9} and \eqref{01.8},
	\begin{align}\label{3}
	\begin{split}
	K(x)A(t)\leq& C_3V(x)\Phi(t)+C_3\Phi_{*}(t),\;\;\;\;\forall t>0\;\text{ and }\;x\in\mathbb{R}^N
	\end{split}
	\end{align}
	with $C_3=\max\{1,C_2\lVert K\lVert_{\infty}, C_2\left\lVert\frac{K}{V}\right\lVert_{L^\infty( B_{r}(0))}\}$. By the inequalities \eqref{2} and \eqref{3}, we get
	{\small\begin{align*}
		\int_{\mathbb{R}^{N}}K(x)A\left(\dfrac{|u|}{C_3\lVert u\lVert_{E}+C_1\lVert u\lVert_{E}}\right)dx\leq C_3	\int_{\mathbb{R}^{N}}V(x)\Phi\left(\dfrac{|u|}{\lVert u\lVert_{V,\Phi}}\right)dx+C_3	\int_{\mathbb{R}^{N}}\Phi_{*}\left(\dfrac{|u|}{\lVert u\lVert_{\Phi_{*}}}\right)dx\leq C_4
		\end{align*}}
	\noindent where $C_4$ is a positive constant that does not depend on $u$. So we can conclude that $E\subset L^{A}_{K}(\mathbb{R}^{N})$.
	
Now, consider $ (v_n) $ a bounded sequence in $ E $. To see that the operator $
	i: E\longrightarrow L^{A}_{K} (\mathbb{R}^{N})$
	is compact just prove that $(v_n)$ has a convergent subsequence on $L^{A}_{K}(\mathbb{R}^{N})$. Since $(v_{n})$ is bounded in $E$, we have that $(v_n)$ is bounded in $D ^{1,\Phi}(\mathbb{R}^{N})$, so there is $ u \in E $ such that $ v_{n}\xrightharpoonup[\quad]{\ast}v$ in $D ^{1,\Phi}(\mathbb{R}^{N})$, or equivalently $w_{n}\xrightharpoonup[\quad]{\ast}0$ in $D ^{1,\Phi}(\mathbb{R}^{N})$, where $w_n = v_n-v$. By the limitation of $(v_n)$ in $E$ and $\Phi,\Phi_{*}\in(\Delta_{2})$, there is $M_1>0$ such that
	\begin{align}\label{5.6}
	\int_{\mathbb{R}^{N}}V(x)\Phi(|w_n|)dx\leq M_1,\;\text{ and }\;
	\int_{\mathbb{R}^{N}}\Phi_{*}(|w_n|)dx\leq M_1,\;\forall n\in\mathbb{N}.
	\end{align}
	Thus, by \eqref{01.9} and \eqref{5.6}, we obtain
	\begin{align}\label{5.8}
	\int_{B_r (0)^{c}}K(x)A(|w_n|)dx\leq \varepsilon M_1,\;\;\;\;\forall n\in\mathbb{N}.
	\end{align}
	Again, we use the Corollary \ref{0.8} and the fact that $ E $ is compactly embedded in $ L^{\Phi}_{loc}(\mathbb {R} ^ {N}) $ to ensure the existence of a subsequence of $(v_n)$, still denoted by itself, such that
	\begin{align*}
	v_n\longrightarrow v\;\;\;\;\text{ in }\;\;\;\;L^{\Phi}(B_{r}(0)).
	\end{align*}
	Thus, there is a subsequence of $(v_n)$, still denoted by itself, that such
	\begin{align*}
	v_n(x)\longrightarrow v(x)\;\;\;\;a.e.\;\text{ in }\;B_{r}(0),
	\end{align*}
	that is,
	\begin{align*}
	w_n(x)\longrightarrow 0\;\;\;\;a.e.\;\text{ in }\;B_{r}(0).
	\end{align*}
	Consider the functions 
	$P:\mathbb{R}\longrightarrow\mathbb{R}$ and $Q:\mathbb{R}\longrightarrow\mathbb{R}$ given by
	\begin{align*}
	P(t)=A(|t|)\;\;\;\;\text{ e }\;\;\;\;Q(t)=\Phi_{*}(|t|).
	\end{align*}
	Clearly $P$ and $Q$ are continuous, moreover
	\begin{align*}
	\lim_{|t|\rightarrow+\infty}\dfrac{P(t)}{Q(t)}=0.
	\end{align*}
	Finally, it follows from the limitation of $(v_n)$ in $E$ that there is $C_1>0$, such that
	\begin{align*}
	\int_{\mathbb{R}^{N}} Q(w_n)dx\leq \int_{\mathbb{R}^{N}}\Phi_{*}(|w_n|)dx<C_1.
	\end{align*}
	Then, by compactness Lemma of Strauss [\citenum{Strauss}, Theorem A.I, p. $338$],
	\begin{align*}
	\int_{B_r (0)}P(w_n)dx\longrightarrow 0.
	\end{align*}
	Therefore,
	\begin{align}\label{5.9}
	\lim_{n\rightarrow\infty}\int_{B_r (0)}K(x)A(|w_n|)dx=0.
	\end{align}
	By \eqref{5.8} e \eqref{5.9}, we have
	\begin{align*}
	\limsup_{n \to \infty}  \int_{\mathbb{R}^{N}}K(x)A(|w_n|)dx	\leq&\limsup_{n\rightarrow\infty}\int_{B_r (0)}A(|w_n|)dx\\
	&+\limsup_{n \to \infty}  \int_{B_r (0)^{c}}K(x)A(|w_n|)dx\\
	\leq& \varepsilon(C M_1+1).
	\end{align*}
	By the arbitrariness of $\varepsilon>0$, it follows that
	\begin{align*}
	\lim_{n\rightarrow\infty} \int_{\mathbb{R}^{N}}K(x)A(|w_n|)dx=0.
	\end{align*}
	As $ A \in (\Delta_{2}) $, 
	\begin{align*}
	w_n\longrightarrow0\;\;\;\;\text{ in }\;\;\;\;L^{A}_{K}(\mathbb{R}^{N}),
	\end{align*}
	in other words
	\begin{align*}
	v_n\longrightarrow v\;\;\;\;\text{ in }\;\;\;\;L^{A}_{K}(\mathbb{R}^{N}),
	\end{align*}
	which completes the proof.
	
	\qed

Note that the condition $(f_6)$ implies that
$\displaystyle\lim_{t\rightarrow+\infty}\frac{f(t)}{\phi_{*}(t)t}=0$. Then, by the conditions $(f_5)$ or $(f_6)$, given $\varepsilon>0 $ there exists $\delta_0>0$, $ \delta_1>0$ and $C_{\varepsilon}>0$ such that
\begin{align}\label{44.0}
K(x)|f(t)|\leq \varepsilon K(x)a(t)t+\varepsilon \lVert{K}\lVert_{\infty}\phi_{*}(t)t+C_{\varepsilon}K(x)\phi_{*}(t)t\chi_{[\delta_0, \delta_1]}(t),
\end{align}
for all $t\geq0$ and $x\in\mathbb{R}^N$. This inequality together with Lemma \ref{44.01} yields that the functional $\mathcal{F}:E\longrightarrow\mathbb{R}$, given by
\begin{equation}
\mathcal{F}(u)=\int_{\mathbb{R}^{N}}K(x)F(u)dx
\end{equation}
is well defined and belongs to $ C^1(E,\mathbb R)$ with derivative
$$
\mathcal{F}'(u)v=
\int_{\mathbb R^N}K(x)f(u)vdx, \quad \forall u,v \in E.
$$
Therefore, we can conclude that the energy functional $J:E\longrightarrow\mathbb{R}$ associated to problem $(P)$, which is given by
$$
J(u)=\int_{\mathbb{R}^{N}} \Phi(|\nabla u|)dx+\int_{\mathbb{R}^{N}}V(x)\Phi(|u|)dx-\int_{\mathbb{R}^{N}}K(x)F(u)dx.
$$
is a continuous and Gateaux-differentiable functional such that $J':E\longrightarrow E^*$ given by
\begin{align*}
J'(u)v=\int_{\mathbb{R}^{N}} \phi(|\nabla u|)\nabla u\nabla v dx+\int_{\mathbb{R}^{N} } V(x)\phi(|u|)uv dx-
\int_{\mathbb R^N}K(x)f(u)vdx
\end{align*}
is continuous from the norm topology of	$E$ to the weak$^*$-topology of $E^*$.
 Once that we intend to find nonnegative solutions, we will assume that
\begin{align}\label{44.22}
f(s)=0,\;\;\;\forall s\in (-\infty,0].
\end{align}
From \eqref{44.0} and $(f_3)$, it follows that $J$ satisfies the geometry of the mountain pass. Hence, there is a Cerami sequence $(u_n)\subset E$, such that,
\begin{align}\label{44.1}
J(u_n)\longrightarrow c\;\;\;\;\text{ and }\;\;\;\;
(1+\lVert u_n\lVert)\lVert J'(u_n)\lVert_{*}\longrightarrow 0
\end{align}
where $c$ is the mountain pass level given by
\begin{align*}
c = \inf_{\gamma \in \Gamma} \max_{t \in [0,1]} J (\gamma(t))
\end{align*}
with
\begin{align*}
\Gamma = \{\gamma \in C([0,1], X) : \; \gamma(0) = 0 \; \text{ and } \; J(\gamma(1)) \leq 0 \}.
\end{align*}
As in the previous section, the above sequence is obtained from the Ghoussoub-Preiss theorem, see [\citenum{Motreanu}, Theorem $5.46$].

In order to show that the Cerami sequence obtained in \eqref{44.1} is bounded, we present the following result.

\begin{lemma}\label{6.77}
	Let $(v_n)$ be a bounded sequence in $E$ such that $v_{n}\xrightharpoonup[\quad]{\ast} v$ in $D^{1,\Phi}(\mathbb{R}^{N})$ . Suppose that $f$ satisfies $(f_5)$ or $(f_6)$, then
	\begin{align}\label{6.771}
	\lim_{n\rightarrow\infty}\int_{\mathbb{R}^{N}}K(x)F(v_n)dx=\int_{\mathbb{R}^{N}}K(x)F (v)dx,
	\end{align}
	\begin{align}\label{6.772}
	\lim_{n\rightarrow\infty}\int_{\mathbb{R}^{N}}K(x)f(v_n)v_ndx=\int_{\mathbb{R}^{N}}K(x)f (v)vdx
	\end{align}
	and
	\begin{align}\label{6.773}
	\lim_{n\rightarrow\infty}\int_{\mathbb{R}^{N}}K(x)f(v_n)\psi dx=\int_{\mathbb{R}^{N}}K(x)f (v)\psi dx,\;\;\;\forall \psi \in C^{\infty}_{0}(\mathbb{R}^{N}).
	\end{align}
\end{lemma}

\noindent {\bf{Proof:}} As in \eqref{44.0}, given $\varepsilon>0$, there exists $\delta_0>0$, $ \delta_1>0$ $C_1>0$ and $C_\varepsilon>0$ such that
\begin{align}\label{44.2}
K(x)|f(t)|\leq C_1K(x)a(t)t+\varepsilon \lVert{K}\lVert_{\infty}\phi_{*}(t)t+C_{\varepsilon}K(x)\phi_{*}(t)t\chi_{[\delta_0,\delta_1]}(t),
\end{align}
 for all $t\geq0$ and $x\in\mathbb{R}^N$.
By the condition $(K_3)$, there is $r_0>0$ sufficiently large satisfying
\begin{align*}
K(x)A(t)\leq \varepsilon\left( V(x){\Phi(t)}+{\Phi_{*}(t)}\right),\;\;\;\;\forall t>0\;\text{ and }\;|x|\geq r_0.
\end{align*}
From the above inequalities, we have
\begin{align}\label{44.3}
K(x)F(t)\leq  \varepsilon C_{1}V(x)\Phi(t)+\varepsilon C_2\Phi_{*}(t)+C_\varepsilon K(x)\Phi_{*}(\delta_1)\chi_{[\delta_0,\delta_1]}(t),\;\forall t>0\text{ and }|x|\geq r_0.
\end{align}
Repeating the same arguments used in the proof of Proposition 2.1, it follows that
\begin{align}\label{44.4}
\limsup_{n\rightarrow+\infty} \int_{B_{r_0}^c (0)}K(x)F(v_{n})dx\leq \varepsilon C_3,
\end{align}
for some constant $C_3>0$ that does not depend on $n$ and $\varepsilon$. On the other hand, the compactness lemma of Strauss [\citenum{Strauss}, Theorem A.I, p. 338], guarantees that
\begin{align}\label{44.5}
\lim_{n\rightarrow+\infty} \int_{B_{r_0} (0)}K(x)F(v_{n})dx=\int_{B_{r_0} (0)}K(x)F(v)dx.
\end{align}
In light of this, we can conclude that
	\begin{align*}
	\lim_{n\rightarrow+\infty}\int_{\mathbb{R}^{N}}K(x)F(v_{n})dx=\int_{\mathbb{R}^{n}}K(x)F (v)dx.
	\end{align*}
In the same way, we can get the limit \eqref{6.772}. Related the limit \eqref{6.773}, it follows directly from the condition $(f_5)$ or $(f_6)$ together with a version of the compactness lemma of Strauss for non-autonomous problem.

\qed

Repeating the same arguments used in the proof of Lemma \ref{5} and of Proposition \ref{5.}, it follows that the Cerami sequence $(u_n)$ given in \eqref{44.1} is bounded, up to some subsequence, we can assume that there is $u \in   E$ such that
\begin{align}\label{44.8}
u_{n}\xrightharpoonup[\quad]{\ast} u\;\;\text{ in }D^{1,\Phi}(\mathbb{R}^{N}),\;\;\;\text{ and }\;\;\;
u_{n}(x)\longrightarrow u(x)\;\;a.e.\;\;\mathbb{R}^{N}.
\end{align}
As in the previous section, we can conclude that $u$ is a nonnegative solution and $u\neq0$. It remains to study the sign of the nonnegative solution for the problem $(P)$.

\subsection{\textbf{Proof of Theorem \ref{Teo3}}} Assuming the assumptions of Theorem \ref{Teo3}, the existence of a nontrivial weak solution for problem $(P)$ is posed by the arguments presented above. In this subsection, as in the subsection \ref{sec1}, we heavily use the fact that $\Phi\in \mathcal{C}_m$ to study the regularity and positiveness of nonnegative solutions of the problem $(P)$, therefore showing Theorem \ref{Teo3}.

We begin by presenting a technical result, but crucial for the study that follows.
\begin{lemma}\label{44.7}
	Let  $u\in E$ be a nonnegative solution of $(P)$, $x_0 \in  \mathbb{R}^N$ and $R_0>0$. Then 
	\begin{align*}
	\int_{\mathcal{A}_{k,t}} |\nabla u|^m dx\leq C \left( \int_{\mathcal{A}_{k,s}}\left|\dfrac{u-k}{s-t}\right|^{m^*}dx +(k^{m^*}+1) |\mathcal{A}_{k,s}|\right)
	\end{align*}
	where $0<t<s<R_0$, $k>1$, $ \mathcal{A}_{k,\rho}=\{x\in B_{\rho}(x_0): u(x)>k\}$ and $C>0$ is a constant that does not depend on $k$.
\end{lemma}

\noindent{\bf{Proof:}} Let $u\in E$ be a weak solution nonnegative of $(P)$ and $x_0\in \mathbb{R}^{N}$. Moreover, fix $0 < t < s < R_0 $. By the condition $(f_5)$ together with the Remark \ref{5.02}, given $\varepsilon>0 $ there exists $C_{\varepsilon}>0$ such that
\begin{align}\label{44.6}
K(x)|f(t)|\leq \varepsilon C_1V(x)\phi(t)t+C_{\varepsilon}\phi_{*}(t)t,\;\;\;\forall x\in B_{R_0}(x_0)
\end{align}
where $ C_1=\Big\lVert\frac{K}{V}\Big\lVert_{L^{\infty} ({B_{R_0}(x_0)})}$. Now, we can repeat the same arguments used in the proof of Lemma \ref{22.6}.

\qed

We consider $u\in E$ a nonnegative solution of $(P)$, repeating the same arguments used in the proof of Lemma \ref{22.21}, it follows that $u\in C^{1,\alpha}_{loc} (\mathbb{R}^N)$. Also, $u>0$ in $\mathbb{R}^N$.Since the proof of this fact is done in an analogous way to Corollary \eqref{22.2}, we will omit their respective proof here.

\subsection{\textbf{Proof of Theorem \ref{Teo3.1}}} Here, we are considering the assumptions of Theorem \ref{Teo1.1}, the arguments at the beginning of this section 5, guarantees the existence of a nontrivial weak solution for problem $(P)$. Then, to prove the Theorem \ref{Teo1.1}, it is enough to show the regularity and positiveness of nonnegative solutions of the problem $(P)$.
In this case, by the condition $(f_6)$ together with the Remark \ref{5.02}, given $\varepsilon>0 $, $x_0\in\mathbb{R }^N$ and $R_0>0$, there exists $C_{\varepsilon}>0$ such that
	\begin{align}\label{44.9}
	K(x)|f(t)|\leq \varepsilon C_1V(x)\phi(t)t+C_{\varepsilon}b(t)t,\;\;\;\forall x\in B_{R_0}(x_0)
	\end{align}
where $ C_1=\Big\lVert\frac{K}{V}\Big\lVert_{L^{\infty} ({B_{R_0}(x_0)})}$. Therefore, without great difficulties, we can show the following result:
	\begin{lemma}\label{44.10}
		Let  $u\in E$ be a nonnegative solution of $(P)$, $x_0 \in  \mathbb{R}^N$ and $R_0>0$. Then 
		\begin{align*}
		\int_{\mathcal{A}_{k,t}} |\nabla u|^\ell dx\leq C \left( \int_{\mathcal{A}_{k,s}}\left|\dfrac{u-k}{s-t}\right|^{\ell^*}dx +(k^{\ell^*}+1) |\mathcal{A}_{k,s}|\right)
		\end{align*}
		where $0<t<s<R_0$, $k>1$, $ \mathcal{A}_{k,\rho}=\{x\in B_{\rho}(x_0): u(x)>k\}$ and $C>0$ is a constant that does not depend on $k$.
\end{lemma}
%
%
We consider $u\in E$ a nonnegative solution of $(P)$, the preceding lemma guarantees that $u\in C^{1,\alpha}_{loc} (\mathbb{R}^N)$. Also, $u>0$ in $\mathbb{R}^N$. Since the proof of these facts are made in an analogous way to Lemmas \ref{22.21} and Corollary \eqref{22.2}, we will omit their respective proof here.
\begin{remark}
	In the Theorems \ref{Teo3} and \ref{Teo3.1}, assuming that $\displaystyle\inf_{x\in\mathbb{R }^N}V(x)>0$, the nonnegatives solution for $(P )$ and they go to zero at infinity, that is, $\displaystyle\lim_{|x|\rightarrow\infty} u(x)=0$.
\end{remark}

\section{Zero mass case}

In this section, we consider the zero mass case, and give the proof of Theorem \ref{Teo2}. Here, the coefficient $K$ only satisfies the conditions $(K_0)$ and $(K_1)$, furthermore, we are considering $V=0$. For $(f_7)$, there is $C>0$ such that
\begin{align}\label{2.0}
	K(x)|f(t)|\leq C\phi_{*}(|t|)|t|,\quad x\in\mathbb{R}^N\;\text{ and }\;t\in\mathbb{R}.
\end{align}

The inequality above yields that $\mathcal{F}:D^{1,\Phi}(\mathbb{R}^N)\longrightarrow\mathbb{R}$ defined by
\begin{equation}
\mathcal{F}(u)=\int_{\mathbb{R}^{N}}K(x)F(u)dx
\end{equation}
is well defined, $\mathcal{F}\in C^{1}(D^{1,\Phi}(\mathbb{R}^N), \mathbb{R})$ with derivative
$$
\mathcal{F}'(u)v=
\int_{\mathbb R^N}K(x)f(u)vdx, \quad \forall u,v \in D^{1,\Phi}(\mathbb{R}^N).
$$

As in the previous sections, the energy functional $J:D^{1,\Phi}(\mathbb{R}^N)\longrightarrow\mathbb{R}$ associated with the problem $(P)$, which is given by
\begin{align*}
J(u)=\int_{\mathbb{R}^{N}} \Phi(|\nabla u|)dx-
\int_{\mathbb R^N}K(x)F(u)dx, 
\end{align*}
is continuous and Gateaux-differentiable with derivative $J': D^{1,\Phi}(\mathbb{R}^N)\longrightarrow (D^{1,\Phi}(\mathbb{R}^N))^*$ defined by
\begin{align*}
J'(u)v=\int_{\mathbb{R}^{N}} \phi(|\nabla u|)\nabla u\nabla v dx-
\int_{\mathbb R^N}K(x)f(u)vdx
\end{align*}
 continuous from the norm topology of	$D^{1,\Phi}(\mathbb{R}^N)$ to the weak$^*$-topology of $ (D^{1,\Phi}(\mathbb{R}^N))^*$. 
 
 Once that we intend to find nonnegative solutions for the problem $(P)$, we will assume that
 \begin{align}\label{2.22}
 f(s)=0,\;\;\;\forall s\in (-\infty,0].
 \end{align}
 
 We will say that $u\in D^{1,\Phi}(\mathbb{R}^N)$ is a critical point for the functional $J:D^{1,\Phi}(\mathbb{R}^N)\longrightarrow\mathbb{R} $ if
 \begin{align}\label{2.15}
 Q(v)-Q(u)\geq\int_{\mathbb{R}^{N}}{K(x)f(u)(v-u)}dx,\;\;\;\;\forall v\in D^{1,\Phi}(\mathbb{R}^N).
 \end{align}
 
Similar to the proof of [\citenum{AlvesandLeandro}, Lemma 4.1], we give the following result. 

 \begin{proposition} \label{2.16}
 	If $u\in D^{1,\Phi}(\mathbb{R}^N)$ is a critical point of $J$ in $D^{1,\Phi}(\mathbb{R}^N)$, then u is a weak solution to $(P)$.
 \end{proposition}

Also due to the inequality \eqref{2.0}, it follows that $J$ satisfies the geometry of the mountain pass. Therefore, by [\citenum{Motreanu}, Theorem $5.46$], there is a sequence $(u_n)\subset D^{1,\Phi}(\mathbb{R}^N)$ such that
\begin{align}\label{2.1}
J(u_n)\longrightarrow c\;\;\;\;\text{ and }\;\;\;\;
(1+\lVert u_n\lVert)\lVert J'(u_n)\lVert\longrightarrow 0.
\end{align}
where $c$ is the mountain pass level given by
\begin{align*}
c = \inf_{\gamma \in \Gamma} \max_{t \in [0,1]} J (\gamma(t))
\end{align*}
with
\begin{align*}
\Gamma = \{\gamma \in C([0,1], X) : \; \gamma(0) = 0 \; \text{ and } \; \gamma(1) = e \}.
\end{align*}
Now, similar to the proof of Lemma \ref{5}, we give the following Lema.

\begin{lemma}\label{66}
Let $(u_n)$ the Cerami sequence given in \eqref{2.1}. There is a constant $M>0$ such that $J(t u_n)\leq M$ for every $t\in [0,1]$ and $n\in\mathbb{N}$.
\end{lemma}

\begin{proposition}
	The Cerami sequence $(u_n)$ given in \eqref{2.1} is bounded in $D^{1,\Phi}(\mathbb{R}^N)$.
\end{proposition}
\noindent{\bf{Proof:}} Suppose for contradiction that, up to a subsequence, $\lVert u_n\lVert_{D^{1,\Phi}(\mathbb{R}^N)}\longrightarrow\infty$ and set $w_n =\frac{u_n } {\lVert \nabla u_n \lVert_{\Phi}}$. Since $\lVert \nabla w_n \lVert_{\Phi}=1 $, there is $w \in D^{1,\Phi}(\mathbb{R}^N)$ such that  $ w_n \xrightharpoonup[\quad]{\ast} w $ in $D^{1,\Phi}(\mathbb{R}^N)$ and $ w_n(x)\rightarrow w(x) $ almost everywhere in $\mathbb{R}^N$. As in the proof of Proposition \ref{5.}, for each $\tau >0$, there is $\xi>0$ sufficiently large such that
\begin{align*}
1\geq\tau \int_{\Omega\cap\{|u_n|\geq \xi\}} K(x)|w_n|^{m}dx,\;\;\;\forall n\geq n_0
\end{align*}
for some $n_0\in\mathbb{N}$ where $\Omega=\{x\in\mathbb{R}^{N}:\;w(x)\neq 0\}$. By Fatou's Lemma
\begin{align*}
1\geq\tau\int_{\Omega} K(x)|w|^{m}dx,\;\;\text{ for all }\tau>0.
\end{align*}
Therefore,
\begin{align*}
\int_{\Omega} K(x)|w(x)|^{m}dx=0.
\end{align*}
By $(K_0)$, we have $K(x)>0$ almost everywhere in $\mathbb{R}^N$, therefore $|\Omega|=0$, showing that $w= 0$.

Note that for every $\sigma>1$, there is $n_0\in \mathbb{N}$ such that $\dfrac{\sigma}{\lVert \nabla u_n\lVert_{\Phi}}\in[0,1]$, for all $n\geq n_0.
$
Given this, we get
\begin{align*}
J(t_n u_n)\geq\sigma\int_{\mathbb{R}^{N}}\Phi(|\nabla w_n|)dx-\int_{\mathbb{R}^{N}}K(x)F(\sigma w_n)dx=\sigma-\int_{\mathbb{R}^{N}}K(x)F(\sigma w_n)dx,\;\;\;\forall n\geq n_0.
\end{align*}
We claim that
\begin{align}\label{2.3}
\lim_{n\rightarrow\infty} \int_{\mathbb{R}^{N}}K(x)F(\sigma w_n)dx=0,
\end{align}
We postpone for a moment the verification of \eqref{2.3}. But if it were true, one would get
\begin{align*}
\liminf_{n\rightarrow\infty} J(t_n u_n)\geq \sigma,\;\text{ for every }\sigma\geq 1,
\end{align*}
which constitutes a contradiction with Lemma \ref{66}, once that $(J(t_n u_n))$ is bounded from above. Therefore $(u_n)$ is bounded.

Now, it remains to prove the limit \eqref{2.3}. From  $(f_7)$, given $\varepsilon>0$, there exists $\delta_0>0$, $ \delta_1>0$ and $C_1>0$ such that
\begin{align}\label{2.4}
|f(t)|\leq \varepsilon\phi_{*}(|t|)|t|+C_1\chi_{[\delta_0,\delta_1]}(t),\;\;\text{ for }t\in\mathbb{R}.
\end{align}
Note that $F_n = \{x \in \mathbb{R}^N : |\sigma w_n(x)| \geq \delta_0\}$ is such that
\begin{align*}
	\Phi_{*}(\delta_0)|F_n|\leq \int_{F_n}\Phi_{*}(|\sigma w_n(x)|)dx\leq \int_{\mathbb {R}^{N}} \Phi_{*}(|\sigma w_n(x)|)dx\leq C_2\int_{\mathbb {R}^{N}} \Phi_{*}(|w_n(x)|)dx,
\end{align*}
for some constant $C_2>0$ that does not depend on $n$. Since $(w_n)$ is bounded in $D^{1,\Phi}(\mathbb{R}^N)$, there is $C_3>0$ satisfying $\int_{\mathbb {R}^{N}} \Phi_{*}(|w_n|)dx< C_3$, for all $n\in\mathbb{N}$. Thus, $\displaystyle\sup_{n\in\mathbb{N}} |F_n|<+\infty$. From $(K_1)$, we have
\begin{align*}
	\lim_{r\rightarrow+\infty}\int_{F_n\cap B_r^c (0)} K(x)dx=0, \;\text{ uniformly in } n\in\mathbb{N},
	\end{align*}
thus, there is $r_0>0$, so that
\begin{align*}
\int_{F_n\cap B_{{r_{0}}}^c (0)} K(x)dx<\dfrac{\varepsilon}{C_1}, \;\;\;\forall n\in\mathbb{N}.
\end{align*}
By \eqref{2.4}, it follows that 
\begin{align*}
\int_{B_{r_0}^c ( 0)}K(x)F(\sigma w_n)dx&\leq \varepsilon\lVert K\lVert_{\infty} \int_{B_{r_0}^c ( 0)}\Phi_{*}(|\sigma w_n(x)|)dx+ C_1\int_{F_n\cap B_{r_0}^c ( 0)}K(x)dx\\
&\leq \varepsilon (\lVert K\lVert_{\infty}+1),
\end{align*}
for all $n\in\mathbb{N}$. Therefore 
\begin{align}\label{2.5}
\limsup_{n\rightarrow+\infty} \int_{B_{r_0}^c (0)}K(x)F(\sigma w_{n})dx\leq \varepsilon (\lVert K\lVert_{\infty}+1).
\end{align}
On the other hand, using $(f_2)$ and the compactness lemma of Strauss [\citenum{Strauss}, Theorem A.I, p. 338], we can conclude that
\begin{align}\label{2.6}
\lim_{n\rightarrow+\infty} \int_{B_{r_0} (0)}K(x)F(\sigma w_{n})dx=0.
\end{align}
Thus, the limit \eqref{2.3} is obtained from \eqref{2.5} and \eqref{2.6}.

\qed

Since that the Cerami sequence $(u_n)$ given in \eqref{2.1} is bounded in  $D^{1,\Phi}(\mathbb{R}^N)$, by Lemma \ref{0.6}, we can assume that for some subsequence, there is $u \in   D^{1,\Phi}(\mathbb{R}^N)$ such that
\begin{align}\label{2.7}
u_{n}\xrightharpoonup[\quad]{\ast} u\;\;\text{ in }D^{1,\Phi}(\mathbb{R}^{N}),\;\;\;\text{ and }\;\;\;
u_{n}(x)\longrightarrow u(x)\;\;a.e.\;\;\mathbb{R}^{N}.
\end{align}
Given these limits, we can obtain, as in Lemma \ref{0.2}, the inequality
\begin{align}\label{2.8}
	 \int_{\mathbb{R}^{N}} \Phi(|\nabla u|)dx\leq\displaystyle \liminf_{n\rightarrow\infty }\int_{\mathbb{R}^{N}} \Phi(|\nabla u_n|)dx.
\end{align}

To complete the proof of Theorem \ref{Teo2}, we need to show the following limits:
	\begin{align}\label{2.9}
\lim_{n\rightarrow\infty}\int_{\mathbb{R}^{N}}K(x)f(u_n)u_ndx=\int_{\mathbb{R}^{N}}K(x)f (u)udx
\end{align}
and
\begin{align}\label{2.10}
\lim_{n\rightarrow\infty}\int_{\mathbb{R}^{N}}K(x)f(u_n)\psi dx=\int_{\mathbb{R}^{N}}K(x)f (u)\psi dx,\;\;\;\forall \psi \in C^{\infty}_{0}(\mathbb{R}^{N}).
\end{align}
Consider $\varepsilon,\delta_0,\delta_1, C_1>0$ as in \eqref{2.4}. To verify \eqref{2.9}, consider $E_n = \{x \in\mathbb{R}^N : |u_n(x)| \geq \delta_0\}$ which satisfies $\sup_{n} |E_n| <\infty$. From $(K_1)$,  there is $r_0>0$, so that
\begin{align*}
\int_{E_n\cap B_{{r_{0}}}^c (0)} K(x)dx<\dfrac{\varepsilon}{\delta_1C_1}, \;\;\;\forall n\in\mathbb{N}.
\end{align*}
By \eqref{2.4}, it follows that 
\begin{align*}
\int_{B_{r_0}^c ( 0)}K(x)f(u_n)u_ndx&\leq \varepsilon m^*\lVert K\lVert_{\infty} \int_{B_{r_0}^c ( 0)}\Phi_{*}(|u_n|)dx+ C_1\delta_1\int_{F_n\cap B_{r_0}^c ( 0)}K(x)dx\\
&\leq \varepsilon (m^*\lVert K\lVert_{\infty}+1),
\end{align*}
for all $n\in\mathbb{N}$. Therefore 
\begin{align}\label{2.11}
\limsup_{n\rightarrow+\infty} \int_{B_{r_0}^c (0)}K(x)f(u_{n})u_n dx\leq \varepsilon (m^*\lVert K\lVert_{\infty}+1).
\end{align}
On the other hand, using $(f_7)$ and the compactness lemma of Strauss [\citenum{Strauss}, Theorem A.I, p. 338], we can conclude that
\begin{align}\label{2.12}
\lim_{n\rightarrow+\infty} \int_{B_{r_0} (0)}K(x)f(u_{n})u_n dx=0.
\end{align}
Thus, the limit \eqref{2.9} is obtained from {\eqref{2.11} and \eqref{2.12}}. Related the limit \eqref{2.10}, it follows directly from the condition $(f_7)$ together with a version of the compactness lemma of Strauss for non-autonomous problem.

Fix $v\in C^{\infty}_0(\mathbb{R}^N)$. By boundedness of Cerami sequence $(u_n)$, we have $J'(u_n)(v-u_n)=o_n(1)$, hence, since $\Phi$ is a convex function, it is possible to show that
\begin{align}\label{2.13}
\begin{split}
\int_{\mathbb{R}^{N}}\Phi(|\nabla v|)dx-\int_{\mathbb{R}^{N}}\Phi(|\nabla u_n|)dx\geq \int_{\mathbb{R}^{N}}K(x)f(u_n)(v-u_n)dx+o_n(1).
\end{split}
\end{align}
From \eqref{2.8}, \eqref{2.13} and \eqref{2.10}, we obtain
\begin{align*}
\int_{\mathbb{R}^{N}}\Phi(|\nabla v|)dx-\int_{\mathbb{R}^{N}}\Phi(|\nabla u|)dx\geq\int_{\mathbb{R}^{N}}K(x)f(u)(v-u)dx.
\end{align*}
By the arbitrariness of $v\in C^{\infty}_{0}(\mathbb{R}^{N})$, it follows that
\begin{align*}
\int_{\mathbb{R}^{N}}\Phi(|\nabla v|)dx-\int_{\mathbb{R}^{N}}\Phi(|\nabla u|)dx\geq\int_{\mathbb{R}^{N}}K(x)f(u)(v-u)dx,\;\;\;\;\forall\;v \in C^{\infty}_{0}(\mathbb{R}^{N}).
\end{align*}
Since $D^{1,\Phi}(\mathbb{R}^{N})=\overline{C^{\infty}_{0}(\mathbb{R}^{N})}^{\lVert\cdot\lVert_{D^{1,\Phi}(\mathbb{R}^{N})}}$ and $\Phi\in( \Delta_{2})$, we conclude that
\begin{align}\label{2.14}
\int_{\mathbb{R}^{N}}\Phi(|\nabla v|)dx-\int_{\mathbb{R}^{N}}\Phi(|\nabla u|)dx\geq\int_{\mathbb{R}^{N}}K(x)f(u)(v-u)dx,\;\;\;\;\forall\;v \in D^{1,\Phi}(\mathbb{R}^{N}).
\end{align}
In other words, $u$ is a critical point of the $J$ functional. And from Proposition  \ref{2.16}, we can conclude that $u$ is a weak solution for $(P)$. Now, we substitute $v=u^+:=\max\{0,u(x)\}$ in \eqref{2.14} and use \eqref{2.22} to get
\begin{align*}
-\int_{\mathbb{R}^{N}}\Phi(|\nabla u^-|)dx\geq\int_{\mathbb{R}^{N}}K(x)f(u)u^-dx=0,
\end{align*}
which leads to
\begin{align*}
\int_{\mathbb{R}^{N}}\Phi(|\nabla u^-|)dx=0,
\end{align*}
which yields $\lVert \nabla u^-\lVert_{\Phi}=0$. Hence, given the inequality \eqref{0.9} we can conclude that $u^-=0$, and therefore, $u$ is a weak nonnegative solution.

It remains to show that $u$ is nontrivial. For this, consider a sequence $(\varphi_k)\subset C^{\infty}_{0}(\mathbb{R}^{N})$ such that $\varphi_k \rightarrow u$ in $D^{1 ,\Phi}(\mathbb{R}^{N})$. Since $(u_n)$ is a limited Cerami sequence, it follows that $J'(u_n)(\varphi_k-u_n)=o_n(1)\lVert \varphi_k\lVert-o_n(1)$. As $\Phi$ is convex, it is possible to show that
{\small\begin{align}\label{2.17}
\int_{\mathbb{R}^{N}}\Phi(|\nabla \varphi_k|)dx-\int_{\mathbb{R}^{N}}\Phi(|\nabla u_n|)dx\geq \int_{\mathbb{R}^{N}} K(x) f(u_n)(\varphi_k-u_n)dx + o_n(1)\lVert \varphi_k\lVert-o_n(1).
\end{align}}
\noindent Because $(\lVert \varphi_k\lVert)_{k\in\mathbb{N}}$ is a bounded sequence, it follows from \eqref{2.9}, \eqref{2.10} and \eqref{2.17} that
\begin{align*}
\int_{\mathbb{R}^{N}}\Phi(|\nabla \varphi_k|)dx\geq \limsup_{n \to \infty}\int_{\mathbb{R}^{N}}\Phi(|\nabla u_n|)dx+\int_{\mathbb{R}^{N}} K(x) f(u)(\varphi_k-u)dx.
\end{align*}
Knowing that $\Phi\in (\Delta_{2})$ and $\varphi_k \rightarrow u$ in $D^{1,\Phi}(\mathbb{R}^{N})$, the last inequality ensures that
\begin{align} \label{2.19}
\int_{\mathbb{R}^{N}}\Phi(|\nabla u|)dx \geq \limsup_ {n \rightarrow \infty} \int_{\mathbb{R}^{N}}\Phi(|\nabla u_n|)dx.
\end{align}
From \eqref{2.8} and \eqref{2.19},
\begin{align}\label{2.20}
 \lim_ {n \rightarrow \infty} \int_{\mathbb{R}^{N}}\Phi(|\nabla u_n|)dx=\int_{\mathbb{R}^{N}}\Phi(|\nabla u|)dx.
\end{align}
Repeating the same arguments used in the proof of limit \eqref{2.3}, that
\begin{align}\label{2.21}
\lim_{n\rightarrow\infty} \int_{\mathbb{R}^{N}}K(x)F(u_n)dx=\int_{\mathbb{R}^{N}}K(x)F(u)dx,
\end{align}
Therefore, by \eqref{2.20} and \eqref{2.21}, we conclude
\begin{align*}
0<c=\lim_{n\rightarrow\infty} J(u_n)=\int_{\mathbb{R}^{N}} \Phi(|\nabla u|)dx-\int_{\mathbb{R}^{N}} K(x)F(u)dx=J(u),
\end{align*}
that is, $u\neq 0$.

\section{Acknowledgments} 
The authors are grateful to the Paraíba State Research Foundation (FAPESQ), Brazil, and the Conselho Nacional de Desenvolvimento Científico e Tecnológico (CNPq), Brazil, whose funds partially supported this paper. 
\section*{Declarations}

\begin{itemize}

	\item\textbf{Ethical Approval }. Not applicable.
	
	\item\textbf{Competing interests}. On behalf of all authors, the corresponding author states that there is no Competing interests.
	
	\item \textbf{Authors Contributions}. Lucas da Silva wrote the main manuscript text under supervision of Marco Souto. All authors reviwed the manuscript.
		
	\item\textbf{Funding}. This study was financed in part by the Paraíba State Research Foundation (FAPESQ), Brazil. Marco Souto was partially supported by the Conselho Nacional de Desenvolvimento Científico e Tecnológico (CNPq), Brazil, 309.692/2020-2.
		
	\item \textbf{Availability of data and materials }. Data sharing and materials not applicable to this article as no datasets and materials were generated or analysed during the current study.

\end{itemize}

\end{document}